\documentclass[a4paper,12pt,reqno]{amsart}
\usepackage{amssymb}
\usepackage{amsmath}
\usepackage{enumitem}
\usepackage{mathrsfs}
\usepackage[all]{xy}
\usepackage{tikz-cd}
\usepackage{mathtools}
\usepackage{hyperref}
\usepackage{url}
\usepackage{bbm}
\usetikzlibrary{arrows}
\usepackage[OT2,T1]{fontenc}
\usepackage[margin=1.2in]{geometry}

\numberwithin{equation}{section} \numberwithin{figure}{section}

\DeclareSymbolFont{cyrletters}{OT2}{wncyr}{m}{n}
\DeclareMathSymbol{\Sha}{\mathalpha}{cyrletters}{"58}

\newcommand{\N}{\mathbb{N}}\newcommand{\Z}{\mathbb{Z}}
\newcommand{\Q}{\mathbb{Q}}
\newcommand{\R}{\mathbb{R}}
\newcommand{\C}{\mathbb{C}}
\newcommand{\F}{\mathbb{F}}

\newcommand{\A}{\mathbb{A}}
\newcommand\GG{\mathbb{G}}

\newcommand\Gm{\GG_\mathrm{m}}
\renewcommand{\P}{\mathbb{P}}

\newcommand{\Adele}{\mathbf{A}}
\renewcommand{\O}{\mathcal{O}}
\newcommand{\dual}[1]{\widehat{#1}}

\DeclareMathOperator{\Hom}{Hom}
\DeclareMathOperator{\Aut}{Aut}

\DeclareMathOperator{\Pic}{Pic}

\DeclareMathOperator{\Gal}{Gal}

\DeclareMathOperator{\Br}{Br}
\DeclareMathOperator{\Bre}{Br^e}
\DeclareMathOperator{\Brun}{Br_{nr}}
\DeclareMathOperator{\HH}{H}

\DeclareMathOperator{\Norm}{\mathbf{N}}
\DeclareMathOperator{\Frob}{Frob}

\DeclareMathOperator{\sep}{sep}

\newcommand{\cycl}{\chi_{\mathrm{cycl}}}
\let\H\undefined
\DeclareMathOperator{\H}{H} 
\DeclareMathOperator{\inv}{inv} 
\let\L\undefined
\DeclareMathOperator{\L}{L}
\DeclareMathOperator{\re}{Re}
\DeclareMathOperator{\im}{Im}
\DeclareMathOperator{\ind}{ind}
\DeclareMathOperator{\rad}{rad}

\newcommand{\GL}{\mathrm{GL}}
\newcommand{\SL}{\mathrm{SL}}
\newcommand{\PSL}{\mathrm{PSL}}

\newtheorem{theorem}{Theorem}[section]

\newtheorem{corollary}[theorem]{Corollary}
\newtheorem{lemma}[theorem]{Lemma}

\theoremstyle{definition}
\newtheorem{definition}[theorem]{Definition}
\newtheorem{conjecture}[theorem]{Conjecture}
\newtheorem*{conjecture*}{Conjecture}
\newtheorem{example}[theorem]{Example}
\newtheorem{remark}[theorem]{Remark}
\newtheorem{question}[theorem]{Question}

\newcommand{\dan}[1]{{\color{blue} \sf $\clubsuit\clubsuit\clubsuit$ Dan: [#1]}}
\newcommand{\tim}[1]{{\color{magenta} \sf $\spadesuit\spadesuit\spadesuit$ Tim: [#1]}}

\begin{document}
	
	\title{The leading constant in Malle's conjecture}
	\author{Daniel Loughran}
	\address{
		Department of Mathematical Sciences \\
		University of Bath \\
		Claverton Down \
		Bath\\ 
		BA2 7AY\\
		UK.}
	\urladdr{https://sites.google.com/site/danielloughran}

	\author{Tim Santens}
	\address{
		University of Cambridge \\ 
		DPMMS \\
		Centre for Mathematical Sciences\\
		Wilberforce Road \\
		Cambridge \\
		CB3 0WB \\ UK}

\begin{abstract}
	We give an overview of a recent conjecture of the authors on the leading constant in Malle's conjecture on number fields of bounded discriminant. This comes from applying the philosophy from Manin's conjecture on rational points of bounded height on Fano varieties to classifying stacks. To make these ideas more accessible we assume no background in algebraic geometry, which requires some new perspectives and alternative approaches to the theory. We also give some new conjectures on multi-heights and Bhargava's heuristics on counting with local conditions imposed. 
\end{abstract}	

\maketitle

\tableofcontents

\section{Introduction}


\subsection{Malle's conjecture}
Let $k$ be a number field. In \cite{Mal02,Mal04}, Malle put forward the following.
\begin{conjecture*}[Malle]
Let $G \subseteq S_n$ be a non-trivial transitive subgroup. Then
$$\#\left\{ K/k :\begin{array}{l}
 [K:k] = n, \Gal(\widetilde{K}/k) \cong G, \\
 |\Norm_{k/\Q} \Delta_{K/k}| \leq B
 \end{array}
 \right\}\sim c_{\mathrm{Malle}}(k,G) B^{a(G)} (\log B)^{b(k,G)-1}$$
for some $c_{\mathrm{Malle}}(k,G) > 0, a(G) > 0, b(k,G) \in \N$.
\end{conjecture*}
In the conjecture $\Delta_{K/k}$ denotes the relative discriminant of $K/k$,  one counts isomorphism classes of fields $K/k$ and $\widetilde{K}$ denotes the Galois closure of $K$ with $\Gal(\widetilde{K}/k) \cong G$ being an isomorphism of permutation groups. Malle gave precise predictions for the invariants $a(G)$ and $b(k,G)$ (we recall these later).

This conjecture has received a significant amount of interest and we highlight some of the  more well-known results: the abelian case \cite{Wri89} and fields of small degree \cite{DH71,CDO02,Bha05,Bha10}.
Despite this, the conjecture has the following issues:
\begin{enumerate}
		\item The exponent $b(k,G)$ of $\log B$ predicted by Malle is wrong in general.
		\item Malle gave no prediction for the leading constant $c_{\mathrm{Malle}}(k,G)$.
		\item The conjecture orders by discriminant which can exhibit pathological properties, and it is convenient to have a framework for different orderings.
		\item For applications, there is a need for finer predictions which impose local conditions, e.g.~counting fields $K/k$ such that $2$ is completely split in $K$.
\end{enumerate}

In recent work \cite{LS24}, we use the theory of algebraic stacks to resolve  these issues and put the counting problem into a single unified framework. This led to precise conjectures for not just the leading constant in (2), but also for different orderings as in (3) and the more refined problem in (4) of counting with local conditions imposed. However certain parts of this work required background knowledge in algebraic stacks which has set a difficult entry bar to verify and understand a conjecture on an explicit counting problem in number theory.

The purpose of this paper is therefore to make the conjectures from \cite{LS24} more accessible assuming no background in algebraic geometry. It is not a survey in the truest sense, as we offer various new perspectives and approaches to the theory from \cite{LS24}, giving alternative proofs and constructions where possible avoiding the theory of stacks. Still: we use the word ``survey'' for want of a better word.


\subsection{Stacks and Malle's conjecture}
The recent progress has been made by viewing Malle's conjecture as analogous to \textit{Manin's conjecture} \cite{FMT89}. This conjecture concerns  rational points of bounded height on Fano varieties. Many researchers had noticed the formal similarities between these two conjectures, however it was Yasuda who was the first to attempt a serious comparison in an unpublished arxiv preprint \cite{Yas15}; Kedlaya \cite[\S10]{Ked07} also proposed that stacks may be useful to study Malle's conjecture.

The first comprehensive attempt at a unification is due to Ellenberg--Satriano--Zureick-Brown \cite{ESZB} who developed a theory of heights on stacks which allowed them to recover Malle's conjecture and Manin's conjecture as special cases. This was swiftly followed by Darda-Yasuda who developed a parallel theory \cite{DYBM} which has the advantage of giving predictions for the power of $\log B$. 

Thankfully, to study Malle's conjecture one does not need the full theory of stacks, but a rather special stack: $BG$, the classifying stack of $G$. To harness this analogy, we view the discriminant as a \emph{height function} on $BG$. A height function is a function which associates to an arithmetic object a positive real number, and is a way to measure the arithmetic complexity of the object. This gives the objects an ordering, with a natural problem to then count the number of objects of bounded height, provided this quantity is finite. 

We explain in more detail below how exactly stacks help, but this new stacky perspective has already led to new developments in Malle's conjecture. This includes a full proof of Malle's conjecture with the predicted leading constant over function fields with large field of constants  for tame groups \cite{San25} and predictions for wild groups \cite{DYBM2}. The paper \cite{LP25} proves the correct lower bound for counting $A_4$-quartic fields of bounded discriminant, by utilising the analogy between Malle's and Manin's conjecture, and transferring existing techniques from Manin's conjecture to Malle's conjecture \cite{FLS18}.  Forthcoming work of the authors with Paterson \cite{LPS26} also gives new predictions for the Cohen--Lenstra heuristics on distributions of (ray) class groups for arbitrary balanced height functions using stacks.

\subsection{Comparison with previous approaches}
We now explain previous approaches and how stacks allow one to resolve the issues highlighted above.  \smallskip

(1) Kl\"{u}ners \cite{Klu05} was the first to find a counter-example to Malle's conjecture for $G$ given by a suitable wreath product. Further counter-examples were found by various authors \cite{KP25,Wan24}. They all have the following similar flavour: There is a subfield $K$ of some cyclotomic extension of $k$, such that the collection of field extensions $k \subset K \subset L$ containing $K$ dominates the counting problem and provides a larger power of $\log B$. T\"{u}rkelli \cite{Tur15} put forward a suggested fix which takes into account the contributions from these cyclotomic subfields separately, however Wang \cite{Wan24} found a counter-example and modified the fix.

We build on this by giving precise predictions, including the leading constant, for the number of fields with a given Galois group which contain any given subfield. We view these additional subfamilies of number fields which need to be treated separately as \textit{accumulating thin subsets} of the classifying stack. To achieve this we work with a more general version of Malle's Conjecture where $G$ admits an action of the Galois group $\Gamma_k$. Here one counts $1$-cocycles with values in $G$, instead of number fields with Galois group $G$. This more general problem is sometimes referred to as the \textit{Twisted Malle's Conjecture} in the literature \cite{Alb21, OA21, AOWW24, AB26, DYTor, Cho25}.  \smallskip

(2) Alberts \cite{Alb23} was the first to consider possible formulae for the leading constant in Malle's conjecture. He did not put forward conjectures, but rather based his investigations on random group models for number fields.

In Manin's conjecture however there is already a prediction for the leading constant due to Peyre \cite{Pey95}  for anticanonical heights, and Batyrev--Tschinkel for more general heights \cite{BT95, BT98}. We take these works as our starting point and adapt to the setting of $BG$ to get a precise prediction, away from a thin set (Conjectures \ref{conj:balanced} and \ref{conj:non_balanced}).

The crucial new ingredient in our predictions is the Brauer group of the stack $BG$: this is required to account for local-to-global obstructions, such as the Grunwald--Wang theorem (see Example \ref{ex:Grunwald-Wang}). Our predictions actually agree with those of Alberts if the Brauer group is trivial. We introduce this Brauer group in a down-to-earth way via central extensions in \S \ref{sec:Brauer}. \smallskip

(3) There is a growing body of evidence that the discriminant is not the most natural way to order number fields for many questions of interest in arithmetic statistics. This was first observed by Wood \cite{Woo10} in the case of abelian extensions, where there can be pathological behaviour in the leading constant when imposing local conditions. Bartel and Lenstra \cite{BL20} also observed that the discriminant is badly behaved when studying distributions of class groups; in particular the Cohen--Lenstra heuristics do not hold in general when ordering by discriminant.

To clarify this we introduced the class of \emph{balanced height functions} in \cite{LS24}. Ordering by such a height function has all the nice properties one would like from an ordering; moreover computational evidence suggests that the discriminant is actually very rarely balanced. The best balanced height function is the \textit{radical discriminant}, namely the product of ramified primes. This may seem a bit naive, but it weights all number fields in the fairest way and gives the most even distribution. \smallskip

(4) In \cite[\S8]{Bha07} Bhargava put forward a heuristic to study the distribution of number fields with local conditions imposed. One hopes that the proportion of such fields changes by a simple ratio according to the imposed local conditions. However this fails in many cases, for example for $D_4$-quartic fields \cite{CDO02}.

To overcome this, we use the notion of equidistribution due to Peyre \cite[\S 3.3]{Pey95}.  It is here that the class of balanced heights returns as it is the class for which one should expect good properties. Provided one removes an accumulating thin set and takes into account additional complications coming from the Brauer group, we are able to extend Bhargava's heuristics to deal with general balanced height functions and put forward precise predictions (Conjectures \ref{conj:equi}, \ref{conj:restricted_ramification}, and \ref{conj:Brauer_spectral}).

\subsection{New contributions}
We now explain some of the new contributions in this survey. As already explained, we give new proofs and perspectives where possible avoiding the  theory of stacks, as well as new examples.

Aside from this, there are the following new contributions to illustrate some natural directions following on from \cite{LS24}.

Firstly, in \S \ref{sec:Bhargava} we extend Bhargava's heuristic on using Dirichlet series to give predictions for number field counts. This leads to Conjecture \ref{conj:Brauer_spectral}, which we call the \emph{Brauer spectral expansion}.

Secondly, in \S\ref{sec:multi_heights} we  analyse multi-heights, building on work of Gundlach \cite{Gun22} and Ellenberg--Venkatesh \cite[\S4.2]{EV05}, where we can also provide precise predictions for some of their heuristics.

\subsection{Prerequisites}
We will assume some familiarity with basic tools in algebraic number theory, including the adeles of a number field, Artin $\L$-functions and Galois cohomology. Specifically non-abelian Galois cohomology in degree $1$ \cite[\S I.5]{Ser02} and the Brauer group of a field \cite[\S1.3.3]{CT21} (but \textit{not} the Brauer group of a variety/stack). We also use some basic category theory.

No knowledge of algebraic stacks, nor even algebraic geometry, is required. The exception is \S\ref{sec:Manin} which is included for motivation only for the interested reader, as well as some related side remarks to motivate certain constructions and definitions. All results, terminology, and constructions in this survey will be explained in an explicit manner, with proofs sketched or deferred to \cite{LS24} or elsewhere as required.


\subsection{Structure of the paper} Since many of our choices are motivated by Manin's conjecture, we begin in \S \ref{sec:Manin} by giving a quick overview of Manin's conjecture for Fano varieties, and relevant generalizations. This is included for motivation only.

In \S \ref{sec:BG} we introduce the key object of study, namely the classifying stack $BG$. We are not assuming any knowledge of stacks, so work with it in a hands-on way via its functor of points.

In \S \ref{sec:heights} we define heights on $BG$, including the key class of \emph{balanced heights}. We then use these in \S \ref{sec:Tamagawa} to define Tamagawa measures on $BG$. These are measures on the adelic points of $BG$ and certain volumes with respect to these measures appear in the leading constant.

In \S \ref{sec:Brauer} we perform a detailed study of the Brauer group of $BG$. Again since we are not assuming any knowledge of stacks we introduce the Brauer group via central extensions of $G$. We work out the resulting theory in this language, including the Brauer--Manin obstruction, which appears in the leading constant.

We then present in \S \ref{sec:leading_constant} our conjectures on the leading constant in Malle's conjecture. We follow this up in \S \ref{sec:equi} with conjectures regarding the finer problem of counting number fields with local conditions imposed. This problem has been considered by Bhargava \cite{Bha07}, and we develop some of his heuristics  into a precise conjecture using what we call the \emph{Brauer spectral expansion} (Conjecture \ref{conj:Brauer_spectral}).

We finish in \S \ref{sec:multi_heights} with a discussion on multi-heights, where we provide precise predictions for heuristics of Gundlach \cite{Gun22} and Ellenberg--Venkatesh \cite[\S4.2]{EV05}.

\subsection{Terminology}  
For a field $k$ we denote by $k^{\sep}$ a choice of separable closure and $\Gamma_k :=\Gal(k^{\sep}/k)$. If $k$ is a global field and $v$ a place of $k$ then we denote by $k_v$ the completion of $k$ at $v$ and $q_v$ the cardinality of its residue field provided $v$ is non-archimedean. We fix an inclusion $k^{\sep} \subset k_v^{\sep}$ inducing a map $\Gamma_v \to \Gamma_k$
.

We say that a finite extension of $k_v$ is \emph{tamely ramified} if its ramification index is coprime to $q_v$. 

A \textit{groupoid} is a category in which every morphism is an isomorphism. An equivalence of groupoids is a functor which induces an equivalence of categories. We call a groupoid \emph{finite} if it has finitely many isomorphism classes of objects and the automorphism group of each object is finite. The correct way to count in a groupoid is via the groupoid cardinality: Let $X$ be a finite groupoid, denote by $[X]$ its set of isomorphism classes. Then the groupoid cardinality of $X$ is defined to be $\#X := \sum_{x \in [X]} 1/|\Aut x|$.

\subsection*{Acknowledgements}
We are very grateful to Julie Tavernier for numerous useful discussions, as well as help with the proofs in \S \ref{sec:Bhargava}. We also thank Brandon Alberts, David Roe, Robert Lemke Oliver, and Melanie Wood for useful comments.

\section{Manin's conjecture} \label{sec:Manin}
	We begin our survey  by giving a very quick overview of the key ideas behind Manin's conjecture \cite{FMT89} relevant to our setting, assuming some modest background in algebraic geometry. 
	
	\subsection{Heights}
	Let $X$ be a Fano variety over a number field $k$, i.e.~a smooth projective variety $X$ whose anticanonical line bundle $\omega_X^{-1}$ is ample. Such varieties are expected to have many rational points. Examples include projective space, but also smooth hypersurfaces in $\P^n$ of degree at most $n$. Given the expectation that such varieties have many rational points, a natural problem is to count them. To do so one needs a measure of the complexity of the rational point, which is provided by a \emph{height function}. We have the naive height on projective space:
	
	\begin{equation} \label{def:naive_height}
		\P^n(k) \mapsto \R_{>0}, \quad (x_0:\cdots:x_n) \mapsto \prod_v \max\{|x_0|_v, \dots, |x_n|_v\}.
	\end{equation}
	Each factor here should be interpreted as a kind of \emph{local height function}. In the special case $k = \Q$ one can choose a representative such that $x_i \in \Z$ and $\gcd(x_0,\dots,x_n) =1$, in which case the formula for the height is simply the familiar $\max\{|x_0|,\dots, |x_n|\}$.
	
	Therefore to each embedding $X \subset \P^n$ we obtain a different choice of height function. However embeddings into a projective space correspond to a very ample line bundle on $X$ together with a choice of generating global sections. This observation, which goes back to Weil, allows a theory of heights based around line bundles and is sometimes called \emph{Weil's height machine} (we emphasise that there are many choices of height associated to a given line bundle).
	
	Essential for counting problems is that the number of rational points of bounded height is \emph{finite}; this is called the Northcott property. The class of height functions with this property are exactly those associated to a so-called \emph{big line bundle} (at least for which the Northcott property holds on a Zariski open subset; the terminology \emph{big} comes from birational geometry).
	
	\subsection{Manin's conjecture}
	
	\begin{conjecture} \label{conj:Manin}
		Let $X$ be a Fano variety over a number field $k$ with $X(k) \neq \emptyset$ and $H$ a height function associated to a big line bundle $L$. Then there exists a thin subset $\Omega \subset X(k)$ such that 
		$$\#\{ x \in X(k) \setminus \Omega : H(x) \leq B\} \sim c(X,H) B^{a(L)} (\log B)^{b(L) - 1}$$
		for some $c(X,H) , a(L), b(L) > 0.$
	\end{conjecture}
	
	The factors $a(L)$ and $b(L)$ have precise predictions due to Manin and Batyrev--Manin \cite{BM90}; the factor $a(L)$ is nowadays called the \emph{Fujita invariant}, due to its importance in birational geometry. 
	
	\subsection{Anticanonical heights} \label{sec:anticanonical}
	Given the many different choices of height available, which height should one take? Well there is a natural class of height functions provided by the \textit{anticanonical height functions}; these are the height functions associated to the anticanonical line bundle. For anticanonical heights the power $b(L)$ of $\log B$ takes largest possible value, so in some sense such heights see the most rational points and treat each rational point equally. The most famous example, and indeed the original motivating example for Manin \cite{FMT89}, was smooth cubic surfaces such as the Fermat cubic
	$$x_0^3 + x_1^3 = x_2^3 + x_3^3 \quad \subset \P^3.$$
	Here a choice of anticanonical height is given by simply the naive height on projective space. The expected thin set to remove is the collection of rational points on the lines, e.g.~remove the ``trivial solutions'' with $x_0 = x_2, x_1 = x_3$.
	
	\subsection{Thin sets} \label{sec:thin_Manin}
	In the original formulation of Conjecture \ref{conj:Manin}, it was expected that only a Zariski closed subset of rational points need be removed. This was inspired by Manin's investigations into cubic surfaces, where he predicted that the lines should be removed to get the correct asymptotic \cite{FMT89}. However Batyrev and Tschinkel \cite{BT96} found counter-examples to this original prediction. Peyre \cite[\S8]{Pey03} suggested that one should remove a \emph{thin subset} of rational points instead; this is now the prevailing philosophy in Manin's conjecture. The definition of thin sets goes back to Serre \cite[\S3.1]{Ser08}, who introduced them to study the inverse Galois problem.
	
	\subsection{The leading constant} \label{sec:leading_constant_Manin}
	We now focus on the most interesting part for us, namely the leading constant $c(X,H)$. Peyre \cite{Pey95} was the first to formulate a conjecture for the leading constant in the case of anticanonical heights. It takes the shape 
	\begin{equation} \label{eqn:Peyre}
	c(X,H) = \alpha(L) \beta(X) \tau_H(X(\Adele_k)^{\Br}).
	\end{equation}
	Here $\alpha(L)$ is Peyre's effective cone constant, $\beta(X) = |\Br X/\Br k|$, $\tau_H$ is Peyre's Tamagawa measure, and $X(\Adele_k)^{\Br}$ the Brauer--Manin set of $X$ with respect to the Brauer group $\Br X$ of $X$. Note that in many texts one sees instead the factor $|\H^1(k,\Pic X_{\bar{k}})| = |\Br_1 X/\Br k|$ where $\Br_1 X$ denotes the \textit{algebraic Brauer group}; that one should take the full Brauer group in Manin's conjecture was first postulated in \cite{LS24}, using computational evidence from Malle's conjecture.  Further evidence for this change appears in forthcoming work of Chan--Koymans--Pagano--Rome--Santens \cite{CKPRS}, where an example of a singular almost Fano variety is given for which there are many more rational points of bounded anticanonical height than would be possible with the $\beta(X) = |\H^1(k,\Pic X_{\bar{k}})|$ version of Peyre's constant.
	
	 Also some texts take instead the measure of the closure of $X(k)$ in $X(\Adele_k)$; we take  the measure of the Brauer--Manin set since it is easier to calculate in practice. More than this, one typically actually proves the formula \eqref{eqn:Peyre} then it is a separate step to show that $X(\Adele_k)^{\Br}$ equals the closure of $X(k)$.
	
	For non-anticanonical heights more care needs to be taken. Firstly, one needs a notion of ``good'' heights. These are now called heights associated to \emph{adjoint rigid} line bundles, with the terminology again coming from birational geometry (this perspective was first put forward by Batyrev-Tschinkel \cite{BT98}). An adjoint rigid line bundle is one for which the \emph{adjoint divisor class} $D:= [\omega_X] + a(L)[L]$ contains a unique effective divisor. For such heights one considers the complement $U:=X \setminus D$ of the adjoint divisor. Then one defines variants $\alpha_U(L)$ and $\tau_{H,U}$ of Peyre's effective cone constant and Tamagawa measure, and the leading constant is given by 
	\begin{equation} \label{eqn:adjoint_rigid}
	c(X,H) = \alpha_U(L) \beta(U) \tau_{H,U}(U(\Adele_k)^{\Br})
	\end{equation}
	where $\beta(U) = |\Br U/\Br k|$ and $U(\Adele_k)^{\Br}$ is the Brauer--Manin set with respect to  $\Br U$. Note in particular that both the Brauer group and topological space $U(\Adele_k)$ depend on the choice of line bundle in the adjoint rigid case. 
	
	Finally, if $L$ is not adjoint rigid, then one reduces to the adjoint rigid case as follows. Associated to $L$ is the so-called \emph{Iitaka fibration} $\phi_L: X \dashrightarrow Y$ to some variety $Y$. The key observation is that the restriction of $L$ to the fibres of $\phi_L$ is now adjoint rigid. Hence we have a precise prediction for the number of points in each fibre. The conjecture is then that the sum over the contributions from all fibres is now convergent. This leads to the prediction
	\begin{equation} \label{eqn:Iitaka_fibration}
	c(X,H) = \sum_{\substack{y \in Y(k) \setminus \phi_L(\Omega)}} c(\phi_L^{-1}(y),H)
	\end{equation}
	where $c(\phi_L^{-1}(y),H)$ is as in \eqref{eqn:adjoint_rigid}.

\section{The stack $BG$} \label{sec:BG}
In this section we introduce the stack $BG$. We begin with some facts on Galois cohomology.
\subsection{$\Gamma_k$-groups and their cohomology}
For this section we will fix a field $k$ of characteristic $p \geq 0$.
It turns out that to get a well-behaved theory in Malle's conjecture one needs to consider groups with a possible Galois action.

\begin{definition}
	Let $k$ be a field. A \emph{$\Gamma_k$-group} $G$ is a group equipped with a (left) $\Gamma_k$-action $\Gamma_k \times G \mapsto G: (\sigma, g) \mapsto \sigma(g)$. It will be called finite if the underlying group is finite.
	
	A finite $\Gamma_k$-group is \emph{tame} if $|G|$ is coprime to the characteristic $p$. Otherwise we call it \emph{wild}.
\end{definition}
\begin{example} \hfill
	\begin{enumerate}
		\item Any abstract group equipped with the trivial Galois action is a $\Gamma_k$-group.
		\item For all $n$ we have the $\Gamma_k$-group $\mu_n := \{\zeta \in k^{\sep}: \zeta^n = 1\}$ of $n$-th roots of unity.
		\item A Galois module is by definition a $\Gamma_k$-group whose underlying group is commutative.
		\item To any $\Gamma_k$-group $G$ we have its dual $\dual{G}:= \Hom(G, k^{\sep,\times})$ given by its group
		of $1$-dimensional characters. This is also a $\Gamma_k$-group and we have $\dual{\dual{G}} = G$ provided
		$G$ is finite abelian.
		\item Let $K/k$ be a finite separable field extension and $G$ a $\Gamma_K$-group. 
		We define the \emph{Weil restriction}\footnote{This is also known as the induction $\mathrm{Ind}_{\Gamma_K}^{\Gamma_k} G$, see \cite[\S2.1.2]{Stix10}} $\mathrm{R}_{K/k} G$ as follows.
		
		As abstract group we consider the group of $\Gamma_K$-equivariant continuous maps $\mathrm{R}_{K/k} G := \Hom_{\Gamma_K}(\Gamma_k, G)$, where $\Gamma_K$ acts on $\Gamma_k$ by left multiplication. The group structure is given by multiplication on the image. Note that by choosing coset representatives of $\Gamma_k/\Gamma_K$ we get an isomorphism $\mathrm{R}_{K/k} G \cong G^{[K:k]}$ of abstract groups.
		
		The $\Gamma_k$-action is given by the law
		\[
		\Gamma_k \times \mathrm{R}_{K/k} G: (\sigma, f) \mapsto \sigma(f) := \left(\tau \mapsto f(\tau \cdot \sigma)\right).
		\]
		We note that is well-defined since for all $\sigma, \tau \in \Gamma_k$ and $\rho \in \Gamma_K$ we have $f(\rho \cdot \tau \cdot \sigma) = \rho(f(\tau \cdot \sigma))$ as $f$ is $\Gamma_K$-equivariant.
	\end{enumerate}
\end{example}
\begin{remark}
	Given a finite \'etale group scheme $G$ over $k$ we get a corresponding finite $\Gamma_k$-group $G(k^{\sep})$. Moreover, this construction defines an equivalence of categories between finite \'etale group schemes over $k$ and finite $\Gamma_k$-groups. In this paper we will use the language of finite $\Gamma_k$-groups, but note that references may use the equivalent notion of finite \'etale group schemes (as in \cite{LS24}).
\end{remark}


\begin{definition} \label{def:cocycle}
	A $1$-cocycle with values in $G$ is a continuous function $\varphi: \Gamma_k \to G$ such that for all $\sigma, \tau \in \Gamma_k$ we have $\varphi(\sigma \tau) = \varphi(\sigma) \sigma(\varphi(\tau))$. We denote the set of $1$-cocycles by $Z^1(\Gamma_k, G)$.
	
	There is a natural right action of $G$ on $Z^1(\Gamma_k, G)$ given by the conjugation action 
	\[
 	Z^1(\Gamma_k, G) \times G \mapsto Z^1(\Gamma_k, G): (\varphi, g) \mapsto \varphi^g :=  (\sigma \mapsto g^{-1} \varphi(\sigma) \sigma(g)).
	\]
	The non-abelian cohomology set $\H^1(k,G) := Z^1(\Gamma_k, G)/G$ is defined as the quotient for this action. This is an abelian group if $G$ is abelian.
\end{definition}

\begin{example}
	Let $G$ be a finite group with trivial $\Gamma_k$-action. Then $Z^1(\Gamma_k, G) = \Hom(\Gamma_k,G)$ is simply the collection of \textit{continuous} homomorphisms from $\Gamma_k$ to $G$. In this case, the conjugacy action is the natural conjugation action induced by $G$ on $\Hom(\Gamma_k,G)$
\end{example}

\begin{example}\label{ex:Weil_restriction}
	Let $K/k$ be a finite extension and $G$ a $\Gamma_K$-group, there is a natural map $Z^1(k, R_{K/k} G) \mapsto Z^1(K, G)$ given by evaluating at $1$, i.e.~the formula 
	\[
	\varphi \mapsto (\sigma \mapsto (\varphi(\sigma))(1)).
	\]
	We note that this is well-defined as for all $\sigma, \tau \in \Gamma_K$ we have by the definition of $R_{K/k} G$ that 
	\[
	(\varphi(\sigma \tau))(1) = \left(\varphi(\sigma) \cdot \sigma(\varphi(\tau)) \right)(1) = (\varphi(\sigma))(1) \cdot (\varphi(\tau))(\sigma) = (\varphi(\sigma))(1) \cdot \sigma((\varphi(\tau))(1))
	\]
	where the last equality is by the $\Gamma_K$-equivariance of $\varphi(\tau)$.
	
	This map is surjective and even induces a bijection $\H^1(k, R_{K/k} G) \cong \H^1(K, G)$, see Lemma \ref{lem:Weil_restriction}.
\end{example}

\begin{example} \label{ex:inner_twist}
	Let $N \subset G$ be a normal subgroup scheme and $\varphi$ a $1$-cocycle with values in $G$. Then we define
	the \emph{inner twist $N_{\varphi}$ of $N$ by $\varphi$} as follows. As a group we take $N_\varphi$ to be simply
	$N$. We use the Galois action coming from twisting by conjugation by $\varphi$. This is:
	\[
\Gamma_k \times N \mapsto N: (\sigma, n) \mapsto \sigma \cdot n := \varphi(\sigma) \sigma(n) \varphi(\sigma)^{-1}.
\]
\end{example}

If $N = G$ then inner twisting does not change the set of cocycles.
\begin{lemma}\label{lem:inner_twist_cocycles}
	Let $G$ be a $\Gamma_k$-group and $\varphi \in Z^{1}(k, G)$. The following map is well-defined, $G$-equivariant and a bijection
	\[
	Z^1(k, G_{\varphi}) \mapsto Z^1(k, G): \psi \mapsto \psi \cdot \varphi := (\sigma \mapsto \psi(\sigma) \varphi(\sigma)).
	\]
\end{lemma}
\begin{proof}
	That is it is well-defined is the following direct cocycle computation
	\[
	\begin{split}
		(\psi \cdot \varphi)(\sigma \tau) &= \psi(\sigma \tau) \varphi(\sigma \tau) = \psi(\sigma) \varphi(\sigma) \sigma(\psi(\tau)) \varphi(\sigma)^{-1} \varphi(\sigma)  \sigma(\varphi(\tau)) \\ &=  (\psi \cdot \varphi)(\sigma) \sigma((\psi \cdot \varphi)(\tau)).
	\end{split}
	\]
	The $G$-equivariance follows from the following direct cocycle computation
	\[
	(\psi^{g} \cdot \varphi)(\sigma) = g^{-1} \psi(\sigma) \varphi(\sigma) \sigma(g) \varphi(\sigma)^{-1} \varphi(\sigma) = (\psi \cdot \varphi)^{g}(\sigma).
	\]
	It is a bijection as the inverse is given by $\psi' \mapsto\psi' \cdot \varphi^{-1}$ where $\varphi^{-1} \in Z^1(k, G_{\varphi})$ is defined as $\varphi^{-1}(\sigma) = \varphi(\sigma)^{-1}$.
\end{proof}

Inner twists arise naturally via the following fundamental construction, which goes back to 	Serre \cite[\S5.5]{Ser02}.

\begin{lemma}
Let 
	$\varphi \in Z^1(k,G)$. Then the collection of cocycles of $G$ with the same image as $\varphi$
	via $q:G \to G/N$ is naturally parametrised by $Z^1(k,N_\varphi)$. More specifically, the map
	$$Z^1(k,N_\varphi) \to \{f \in Z^1(k,G): q \circ \varphi = q \circ f\},
	\quad \psi \mapsto \psi \cdot \varphi.
	$$
	is a bijection.
\end{lemma}
\begin{proof}
	We may replace $G$ by $G_{\varphi}$ using Lemma \ref{lem:inner_twist_cocycles}. In this case the claim is that $Z^1(k, N_{\varphi}) = \{ f \in Z^1(k, G_{\varphi}): q \circ e_{BG_{\varphi}} = q \circ f\}$. This is obvious.
\end{proof}

\subsubsection{Unramified cohomology}
Assume now that $k$ is a number field.

\begin{definition} \label{def:unramified}
	We call a $\Gamma_k$-group $G$ \textit{unramified} at a place $v$ if the inertia group $I_v \subset \Gamma_{k_v}$ acts trivially on $G$.
	
	Let $v$ be a place at which $G$ is unramified. We call a cocycle $\varphi \in Z^1(k_v,G)$ \textit{unramified} if it factors through $\Gal(k_v^{\textrm{nr}}/k_v) \cong \Gamma_{k_v}/I_v$, where $k_v^{\textrm{nr}}$ is the maximal unramified extension of $k_v$. 
	
	We denote by $Z^1_{\textrm{nr}}(k_v,G) \subset Z^1(k_v, G)$ the set of unramified cocycles.
	We let $\H^1_{\textrm{nr}}(k_v, G):=Z^1_{\textrm{nr}}(k_v,G)/G$ be
	the associated set of unramified cohomology classes.

	A place $v$ is called \emph{good} (with respect to $G$) if it is non-archimedean, $G$ is unramified at $v$ and $G$ is tame at $v$, i.e.~$\gcd(|G|,q_v)=1$. A \emph{bad place} is a non-good place.
\end{definition}

The unramified Galois group $\Gal(k_v^{\mathrm{nr}}/k_v)$ is isomorphic to the free profinite group $\hat{\Z}$. This leads to the following explicit description.

\begin{lemma} \label{lem:unramified_cocycle}
	Let $\sigma \in \Gal(k_v^{\mathrm{nr}}/k_v)$ be a topological generator. If $G$ is a finite unramified $\Gamma_k$-group then the map 
	\[
	Z^1_{\mathrm{nr}}(k_v,G) \to G: \varphi \to \varphi(\sigma)
	\]
	is a bijection.
\end{lemma}
\begin{proof}
	We may identify $\Gal(k_v^{\mathrm{nr}}/k_v) \cong \hat{\Z}$ and $\sigma = 1 \in \hat{\Z}$.
	
	
	If $\varphi \in Z^1_{\mathrm{nr}}(k_v,G)$ then for all $k \in \Z_{> 0}$ we have $\varphi(\sigma^k) = \varphi(\sigma) \cdot \sigma(\varphi(\sigma)) \cdots \sigma^{k - 1}(\varphi(\sigma))$ by the cocycle law. The value of $\varphi(\sigma^k)$ is thus determined by $\varphi(\sigma)$. But the subset $\{\sigma^k : k \in \Z_{> 0} \} \subset \Gal(k_v^{\mathrm{nr}}/k_v) \cong \hat{\Z}$ is dense as $\sigma$ is a topological generator, so $\varphi(\sigma)$ determines all values of $\varphi$, this shows injectivity.
	
	Let now $g \in G$. For all $k \in \Z_{> 0}$ we consider $N_k(g) := g \cdot \sigma(g) \cdots \sigma^{k-1}(g) \in G$. This satisfies the following cocycle law for all $k, \ell \in \Z_{> 0}$.
	\begin{equation}\label{eq:cocycle_type_law}
		N_{k + \ell}(g) = N_k(g) \sigma^k(N_{\ell}(g)).
	\end{equation}
	By finiteness of $G$ there exists $n < m \in \Z_{> 0}$ such that $N_n(g) = N_m(g)$. By \eqref{eq:cocycle_type_law} we thus have $\sigma^n(N_{m - n}(g)) = N_n(g)^{-1} N_m(g) = 1$. We then define a cocycle $\varphi_g$ by the formula
	\[
	\varphi_g: \hat{\Z} \twoheadrightarrow  \Z/(m - n) \Z \to G: \ell \to N_\ell(g).
	\]
	This is well-defined and a cocycle by \eqref{eq:cocycle_type_law}. We note that $\varphi_g(\sigma) = g$ by construction. This implies surjectivity as $g$ was arbitrary.
\end{proof}

\subsection{The stack $BG$}

We define $BG$ via its functor of points over any field $k$.

\begin{definition} \label{def:BG}
	We define $BG(k)$ to be the groupoid whose objects are cocycles with values in $G$,
	and whose morphisms are given by conjugation, as in Definition~\ref{def:cocycle}.
	
	We denote by $BG[k]$ the set of isomorphism classes of objects in $BG(k)$. By construction we have
	$BG[k] = \H^1(k,G)$.
\end{definition}

\begin{example}
If $G$ has trivial Galois action then, by Galois theory, an element of $BG(k)$ corresponds to a Galois extension $L/k$ together with a choice of embedding $\Gal(L/k) \to G$. It is useful for the theory to allow non-surjective homomorphisms. For example if  $v$ is a place of a number field $k$, then the image of a surjective element of $BG(k)$ under the map $BG(k) \to BG(k_v)$ will typically not be surjective.
\end{example}

\begin{example}
	$B \Z/2\Z(k)$ has two types of elements. The first correspond to separable quadratic extensions of $k$. The second is the trivial homomorphism $\Gamma_k \to 0$, which one should view as corresponding to the trivial quadratic algebra $k \times k$.
\end{example}

An important corollary of Lemma \ref{lem:inner_twist_cocycles} is that $BG(k)$ is invariant under inner twists of $G$.
\begin{corollary}\label{cor:inner_twists}
	Let $G$ be a $\Gamma_k$-group and $\varphi \in Z^{1}(k, G)$. Then the formula $\psi \to \psi \cdot \varphi$ induces an equivalence of groupoids $BG_{\varphi}(k) \cong BG(k)$.
\end{corollary}

The following lemma is the crucial property of the Weil restriction.
\begin{lemma}\label{lem:Weil_restriction}
	The map of Example \ref{ex:Weil_restriction} induces an equivalence  $BR_{K/k} G(k) \cong BG(K)$ of groupoids. In particular, it induces a bijection $BR_{K/k} G[k] \cong BG[K]$.
\end{lemma}
\begin{proof}
	If $G$ is abelian then this is Shapiro's lemma. For a proof in the non-abelian setting on the level of isomorphism classes we refer to \cite[Prop.~8]{Stix10}. That the map preserves automorphism groups follows from the same argument as injectivity.
\end{proof}

The classifying stack $BS_n$ has an alternative description which allows one to encode non-Galois extensions of degree $n$. Recall that an \textit{\'etale $k$-algebra} is simply a product of separable field extensions. In the following statements by ``general'' we strictly speaking mean ``away from a thin set'' (see \S\ref{sec:thin} for terminology).

\begin{lemma} \label{lem:S_n}
	The groupoid $BS_n(k)$ is equivalent to the groupoid whose objects are degree $n$ \'etale $k$-algebras and whose morphisms are $k$-algebra isomorphisms. In particular, a general element of $BS_n(k)$ corresponds to a separable field extension of degree $n$.
\end{lemma}
\begin{proof}
	This is a special case of the general phenomenon of cocycles parametrizing twists \cite[\S III.1]{Ser02}. Namely, a $k$-algebra $A$ is a degree $n$ \'etale algebra if and only if $A \otimes_{k} k^{\sep} \cong (k^{\sep})^n$. Moreover, such an isomorphism is unique up to the action of $S_n$ on $(k^{\sep})^n$ which permutes the factors. So the \'etale algebras are exactly the Galois twists of the algebra $k^n$. As a $k$-algebra $k^n$ has automorphism group $S_n$, thus the twists are parametrised by cocycles with values in $S_n$, which is exactly $BS_n(k)$.
\end{proof}

\begin{example}
	Let $G \subset S_n$ be a transitive subgroup. Then the image of the map $BG(k) \to BS_n(k)$ consists of exactly those 
	degree $n$ \'etale $k$-algebras whose Galois group is conjugate to a subgroup of $G$ as a permutation group. In particular, a general element in the image corresponds to a degree $n$ field extension whose Galois group is conjugate to $G$. Such fields are exactly the focus of Malle's original conjecture.
\end{example}

The natural way to count in a groupoid is via the groupoid cardinality. To get to more uniform looking statements, we use the following lemma. 

\begin{lemma} \label{lem:groupoid_count}
	Let $F:Z^1(k,G) \to BG[k]$ be the natural map and $f:BG[k] \to \C$ any function. Then for any
	finite subset $W \subseteq BG[k]$ we have
	$$\sum_{\varphi \in W} \frac{f(\varphi)}{|\Aut \varphi|} = \frac{1}{|G|}\sum_{ \psi \in F^{-1}(W)} f(F(\psi)).$$
\end{lemma}
\begin{proof}
	We may assume $|W| =1$, where it is immediate from
	the 	orbit-stabiliser theorem.
\end{proof}

We also require a notion of integral points on $BG$ when $k$ is a number field. These can be defined provided $G$ is unramified.
\begin{definition} \label{def:integral_point}
	Let $v$ be a place at which $G$ is unramified, in the sense of Definition \ref{def:unramified}.
	We let $BG(\O_v) := Z^1_{\textrm{nr}}(k_v,G) \subset BG(k_v)$ be the corresponding fully faithful subgroupoid	and thus $BG[\O_v]:=\H^1_{\textrm{nr}}(k_v,G)$.
\end{definition}

\subsection{Thin sets} \label{sec:thin}
As explained in \S\ref{sec:thin_Manin}, in Manin's conjecture it is necessary to remove a thin set of rational points to obtain the correct prediction. 

A formal definition of a thin subset in terms of stacks can be found in  \cite[Def.~2.5]{LS24}. We will instead use the following equivalent cocycle-theoretic definition.

\begin{definition} \label{def:thin}
	Let $G$ be a $\Gamma_k$-group, $\varphi \in Z^1(k, G)$ and let $H \subsetneq G_{\varphi}$ be a $\Gamma_k$-subgroup. We define the corresponding \emph{basic thin subset} $\Omega_{\varphi, H} \subset BG[k]$ as the image of the composition
	\[
	BH[k] \to BG_{\varphi}[k] \cong BG[k]
	\]
	where the isomorphism is Corollary \ref{cor:inner_twists}.
	
	A subset $\Omega \subset BG[k]$ is \emph{thin} if it contained in a finite union of basic thin subsets.
\end{definition}

\begin{example} \label{ex:thin} \hfill
	\begin{enumerate}
		\item The collection of non-surjective elements of $BG[k]$ is thin \cite[Lem.~2.7]{LS24}.
		\item Let $N \subsetneq G$ be a non-trivial normal subgroup scheme and $\psi$ a $1$-cocycle with values in $G/N$. 
		We call an element $\varphi \in BG[k]$ a \textit{lift} of $\psi$ if the composition $\Gamma_k \to G \to G/N$
		induced by $\varphi$ equals $\psi$. 
		
		Then the collection of elements $\varphi \in BG[k]$ which lift $\psi$ is thin. Indeed, it is either empty or there exists a lift $\varphi$. Then by Example \ref{ex:inner_twist} the set of lifts is the basic thin subset coming from  $N_{\varphi} \subset G_{\varphi}$.
	\end{enumerate}
\end{example}

This latter example places a geometric structure on solutions to embedding problems. This perspective is taken up in more detail in \cite{LT26}.

\begin{example} \label{ex:wreath}
	Let us consider the wreath product $G \wr S_n$ for some $n$, i.e. $G \wr S_n = G^n \rtimes S_n$ where $S_n$ acts on $G^n$ by permuting the factors.
	
	Let $\varphi \in Z^1(k, S_n)$ correspond to a degree $n$-extension $K/k$ as in Lemma \ref{lem:S_n}, think of $\varphi \in Z^1(k,G \wr S_n)$ using the section $S_n \to G \wr S_n$ coming from the semi-direct product structure. One can then compute that $(G^n)_{\varphi} \cong R_{K/k} G$.
	
	In particular, if $k = \Q$, $G = C_3$, $n = 2$ and $K = \Q(\sqrt{-3})$ then this shows that the set of cyclic cubic extensions of $ \Q(\sqrt{-3})$ forms a thin subset of $B(C_3 \wr C_2)[\Q]$. This is Kl\"uners counterexample \cite{Klu05} to Malle's original conjecture.
\end{example}

\section{Heights} \label{sec:heights}
A function which associates to a number field a real number appears under various names in the literature; for example ``$f$-discriminant'' \cite[\S4.2]{EV05},  “counting function” \cite[\S2.1]{Woo10}, or ``invariant'' \cite{Alb21, Wan24}.

In this work, we use the term \emph{height function}. This choice reflects the usage of height functions to measure the complexity of arithmetic objects and aligns with terminology used elsewhere in arithmetic statistics, such as the height of an elliptic curve or the height of a rational point. It also avoids confusion with invariants under group actions or with the local invariant $\inv_v: \Br k_v \to \Q/\Z$ from class field theory, both of which will appear in our work. More than this, the height functions we define can be interpreted as heights on stacks \cite{DYTor,DYBM}, thus offer a unified approach to numerous problems in arithmetic statistics. 

	Some texts use height functions which are not covered by our setting, for example the $\A^1/\P^1$-height from \cite{P1Height}. To differentiate between these heights in a more general context, one could call the heights from Definition \ref{def:heights} \textit{adelic} or \textit{inertial heights}. Since we do not consider more general heights in this paper, we use the term \emph{height} for brevity.

\subsection{Galois actions on conjugacy classes}
Let $k$ be a field of characteristic $0$ with absolute Galois group $\Gamma_k := \Gal(\bar{k}/k)$ and let $G$ be a finite $\Gamma_k$-group. Fundamental to our framework is the following.

\begin{definition}[Tate twist] \label{def:G(-1)}
	Let $\dual{\Z}(1) := \varprojlim_n \mu_n$. We define
\begin{equation*} 
	G(-1):=\Hom(\dual{\Z}(1), G).
\end{equation*}
We denote by $\mathcal{C}_G:=G(-1)/\mathrm{conj}$ the quotient of $G(-1)$ via the natural conjugacy action of $G$. We let $e \in \mathcal{C}_G$ be the conjugacy class of the identity and write $\mathcal{C}_G^* := \mathcal{C}_G \setminus \{e\}$.
\end{definition}

Both $G(-1)$ and $\mathcal{C}_G$ are finite sets equipped with a $\Gamma_k$-action. Explicitly, the  Galois action is given by
$$ \Gamma_k \times G(-1) \to G(-1), \quad (\sigma,\gamma) \mapsto (\zeta \mapsto \sigma(\gamma(\sigma^{-1}(\zeta)))).$$
 We warn that if $G$ is non-abelian then $G(-1)$ has no natural group structure.  Any element of $G(-1)$ may be represented by a homomorphism $\mu_{\exp(G)} \to G$ where $\exp(G)$ denotes the exponent of $G$.

There is a very closely related Galois set which appears more commonly in the Malle's conjecture literature. Note that the set $G$ admits an action of $\hat{\Z}^\times$ via exponentiation; we call this the \emph{invertible powering action} (this preserves the group structure if and only if $G$ is abelian).

\begin{definition}[Cyclotomic character]
Recall that there is a canonical injection $\Gal(k(\mu_n)/k) \to (\Z/n\Z)^\times$ for which the automorphism $\zeta_n \mapsto \zeta_n^a$ is sent to the class $a \in (\Z/n\Z)^\times$. Taking the limit over all $n$, we obtain the \emph{cyclotomic character} $\cycl: \Gamma_k \to \dual{\Z}^\times$.
\end{definition}

\begin{definition}[Anticyclotomic twist]\label{def:G(cycl)}
Let $G(\cycl^{-1})$ denote the $\Gamma_k$-set with underlying set $G$, but with Galois action twisted by $\cycl^{-1}$:
\begin{equation*} 
\Gamma_k \times G(\cycl^{-1}) \to G(\cycl^{-1}), \quad (\sigma,g) \mapsto \sigma(g)^{\cycl(\sigma)^{-1}}.
\end{equation*}
\end{definition}

The Galois sets $G(-1)$ and $G(\cycl^{-1})$ are isomorphic, albeit non-canonically.
\begin{lemma}\label{lem:Galois_action_on_G(-1)}
Choose a primitive topological generator $(\zeta_n)_n \in \hat{\Z}(1)$. Then the map
$$G(-1) \to G(\cycl^{-1}), \quad (\gamma: \mu_{n} \to G) \mapsto \gamma(\zeta_n)$$
is an isomorphism of $\Gamma_k$-sets which preserves the conjugacy action.
\end{lemma}
\begin{proof}
	It is clearly bijective. Thus it suffices to show that it preserves the Galois
	action. For this we have
	\[(\sigma \cdot \gamma)(\zeta_n) := \sigma(\gamma(\sigma^{-1}(\zeta_n))) = \sigma(\gamma(\zeta_n^{\chi(\sigma^{-1})})) 
	= \sigma(\gamma(\zeta_n))^{\chi(\sigma)^{-1}}. \qedhere \]
\end{proof}

\begin{definition} \label{def:residue_field}
	Let $c \in \mathcal{C}_G$. We define the \emph{residue field} $k(c)$ of $c$ to be
	the fixed field of the stabiliser of $c$ with respect to the action of $\Gamma_k$.
	The isomorphism class of this field only depends on the Galois orbit of $c$.
\end{definition}

\begin{remark}[Sectors]
	In \cite{DYBM,LS24}, the elements of $\mathcal{C}_G/\Gamma_k$,
	i.e.~the Galois orbits of conjugacy classes of $G(-1)$, are called \emph{sectors};
	see \cite[\S4.3]{LS24} for more details.	This terminology comes from string theory.
	It is the sector perspective which generalises to other stacks, 
	and put forward in \cite{DYBM}.
\end{remark}

\subsubsection{Computation via character tables}
If $G$ is constant then one can read off the Galois action on $\mathcal{C}_G$ from the character table of $G$ as follows. Recall \cite[Thm.~24]{Ser77} that if $G$ is a finite group then any character $\chi: G \to \C$ factors through $k(\mu_{\infty}) \subset \C$.

\begin{lemma} \label{lem:complex_character_table}
	The number of Galois orbits of $\mathcal{C}_G$ is equal to the number of Galois orbits of irreducible complex characters of $G$.
\end{lemma}
\begin{proof}
Let $\mathrm{Irr}(G)$ be the set of characters $G \to k(\mu_{\infty})$ of irreducible representations. This is equipped with a natural action of $\Gal(k(\mu_{\infty})/k) \subset \hat{\Z}^{\times}$. This is related to the $\Gal(k(\mu_{\infty})/k)$ action on $G(\cycl^{-1})$ in the following way. 

Let $V$ be the $k$-vector space of class functions $G(\cycl^{-1}) \to k(\mu_{\infty})$ which are $\Gal(k(\mu_{\infty})/k)$-equivariant. Any $\chi \in \mathrm{Irr}(G)$ defines an element of $V$ and they form a basis by \cite[Thm~25]{Ser77}. Consider the action of $\Gal(k(\mu_{\infty})/k)$ on $V$ by only acting on the image. The vector space $V^{\Gal(k(\mu_{\infty})/k)}$ has by the above a basis given by the Galois orbits of $\mathrm{Irr}(G)$. But by definition it also has a basis given by the Galois orbits of the indicator functions of conjugacy classes of $G(\cycl^{-1})$. The result now follows from  Lemma \ref{lem:Galois_action_on_G(-1)}.
\end{proof}

When $k = \Q$ there is another way to find the number of Galois orbits in terms of the \emph{rational character table} of $G$. This is the table whose entries are the sums of the Galois conjugates of the entries of the complex character table (see \cite[\href{https://www.lmfdb.org/knowledge/show/group.rational_character_table}{Rational character table of a group}]{LMFDB}).

\begin{lemma} \label{lem:rational_character_table}
	For $k = \Q$, the number of Galois orbits of $\mathcal{C}_G$ is equal to the number of 
	rows in the rational character table of $G$.
\end{lemma}
\begin{proof}
	Immediate from the definition and Lemma \ref{lem:complex_character_table}.
\end{proof}
The advantage of Lemma \ref{lem:rational_character_table} is that the rational character table is often much smaller than the complex character table, hence easier to work with. Moreover it is usually included in standard online databases, such as the LMFDB.

In the following, we say that the complex character table is \emph{rational} if its entries are rational numbers; equivalently the complex character table equals the rational character table.

\begin{example} \label{ex:Galois_action} \hfill
	\begin{enumerate}
		\item The symmetric group $S_n$ has a rational complex character table so the Galois action on $\mathcal{C}_G$ is trivial.
		\item If $G = \mu_n$ then $G(-1) = C_n$ with trivial Galois action.
		\item If $G = \Z/n \Z$ then $G(-1) = \Hom(\mu_n, \Z/n\Z)$. For example, if $\sigma \in \Gamma_k$ sends a primitive $n$-th root of unity $\zeta$ to $\zeta^q$ for $(n, q) = 1$ then it acts by multiplication by $q^{-1}$ on $G(-1)$.
		\item If $G = A_4$ then $\mathcal{C}_G$ has $4$ elements. There are two rational elements, one of order $1$ and one of order $2$. The other two elements have order $3$ and Galois acts on them by factoring through the quadratic extension $\Q(\sqrt{-3})/\Q$.
	\end{enumerate}
\end{example}

\subsection{Ramification type} \label{sec:ramification_type}
Let $k$ be a number field. Essential to our framework is the notion of \emph{ramification type}, as introduced in \cite[\S7.1]{LS24}. This keeps track of the ramification behaviour of a cocycle, and allows one to compare the ramification of different cocycles using group theoretic data.

As motivation, let $\varphi \in BG(k)$. Recall that to a place $v$ which is unramified in the number field corresponding to $\varphi$, one can associate a conjugacy class of $G$, namely, the conjugacy class of the image of the frobenius element at $v$. The ramification type is an analogue of this construction at the tame places which takes into account the ramification behaviour of the extension, but instead we obtain a conjugacy class of $G(-1)$ rather than a conjugacy class of $G$.

To define this we require the following explicit description of the tame inertia group. Let $I_v$ denote the inertia group at $v$ and $I_v^\mathrm{tame}$ the tame inertia group, i.e.~the Galois group of the maximally tamely ramified extension of $k_v^{\text{nr}}$. 

\begin{lemma} \label{lem:tame_inertia}
	Fix a uniformizer $\pi_v$ of $k_v$. The map
	$$\theta: \varprojlim\limits_{\mathclap{\gcd(q_v,n) = 1}} \,\,\mu_n \to  I_v^{\mathrm{tame}},
	\quad \zeta_n \mapsto ( \pi_v^{1/n} \mapsto \zeta_n \pi_v^{1/n})$$
	is an isomorphism and independent of the uniformizer $\pi_v$.
\end{lemma}
\begin{proof}
 See \cite[\href{https://stacks.math.columbia.edu/tag/09EE}{Tag 09EE}]{stacks-project} and its proof.
\end{proof}

Recall the definition of \emph{good places} from Definition \ref{def:unramified}

\begin{definition} \label{def:ramification_type}
	Let $v$ be a good place. The \textit{ramification type} is the map
	$$\rho_{G,v}: BG(k_v) \to \mathcal{C}_G^{\Gamma_{k_v}},  \quad \varphi_v \mapsto \varphi_v|_{I_v^\mathrm{tame}}.$$	
\end{definition}

Let us describe in more detail this map via the exact sequence
$$0 \to I_v \to \Gamma_{k_v} \to \Gamma_{\F_v} \to 0.$$
Assume first for simplicity that $G$ has trivial Galois action. We restrict $\varphi_v$ to $I_v \subset \Gamma_{k_v}$. As $v$ is tame for $G$, the induced map factors through $I_v^{\mathrm{tame}}$. But by Lemma \ref{lem:tame_inertia} this is canonically isomorphic to a group of roots of unity, thus we obtain an element of $G(-1)$. This is only well-defined up to conjugacy, hence we obtain an element of $\mathcal{C}_G:=G(-1)/\mathrm{conj}$. For general $G$, since $v$ is unramified the inertia group $I_v$ acts trivially on $G$. In which case the cocycle condition implies that the restriction of $\varphi_v$ to $I_v$ is actually a homomorphism; the verification is then analogous.

We next explain why the image of $\rho_{G,v}$ is contained in the Galois invariant part of $\mathcal{C}_G$. Here $\rho_{G,v}$ is essentially the restriction map $\H^1(\Gamma_{k_v},G) \to \H^1(I_v,G)$ on the level of cohomology. Then it is a standard fact in cohomology that the image of the restriction map is contained in $\H^1(I_v,G)^{\Gamma_{k_v}}$; if $G$ is abelian this is built into the inflation-restriction exact sequence \cite[\S 2.6(b)]{Ser02}, whereas for non-abelian $G$ there is an analogous result \cite[\S I.5.8(a)]{Ser02}.




\subsubsection{Categorification} \label{sec:categorification}
We next construct a categorified version of the ramification type which allows for more general applications and a more flexible theory (this can be skipped at first reading). To do so we need to do some more refined Galois theory.

Let $\Gamma_{k_v}^{\mathrm{tame}}$ be the tame Galois group of $k_v$ and let $\Frob_v \in \Gamma_{\F_v}$ be the Frobenius element.

\begin{lemma}\label{lem:lift_of_frobenius}
	Fix a uniformizer $\pi_v$ of $k_v$ and a compatible system of roots $\pi_v^{1/n} \in k_v^{\mathrm{sep}}$ for all $(n, q_v) = 1$. There exists a unique lift of Frobenius $F \in \Gamma_{k_v}^{\mathrm{tame}}$ which fixes all $\pi_v^{1/n}$. We then have $\theta(\Frob_v(\zeta)) = F \theta(\zeta) F^{-1}$ for all $\zeta \in \varprojlim\limits_{\mathclap{\gcd(q_v,n) = 1}} \,\,\mu_n$.
\end{lemma}
\begin{proof}
	Let $k_v^{\mathrm{tame}} \subset \bar{k}_v$ be the maximal tamely ramified extension of $k_v$, so that $\Gal(k_v^{\mathrm{tame}}/k_v) \cong \Gamma_{k_v}^{\mathrm{tame}}$. Let $k_v(\pi_v^{1/\infty}) := \bigcup_{(q_v, n) = 1} k_v(\pi_v^{1/n})$. We have $k_v^{\mathrm{nr}} \cap k_v(\pi_v^{1/\infty}) = k_v$ as it is the intersection of a totally ramified and an unramified extension. On the other hand $\Gal(k_v^{\mathrm{tame}}/k_v(\pi_v^{1/\infty})) \cap I_v^{\mathrm{tame}} = 1$ by Lemma \ref{lem:tame_inertia}.
	
	It follows from these two facts and the Galois correspondence that the composition $\Gal(k_v^{\mathrm{tame}}/k_v(\pi_v^{1/\infty})) \subset \Gamma_{k_v}^{\mathrm{tame}} \to \Gamma_{k_v}^{\mathrm{nr}}$ is an isomorphism. We may take $F$ a lift of Frobenius under this isomorphism.
	
	To show the last equality we write $\zeta = (\zeta_n)_n$ and compute that $(F \theta(\zeta) F^{-1})(\pi_v^{1/n}) = \zeta_n^{q_v} \pi_v^{1/n} = \theta(\zeta_n^{q_v})(\pi_v^{1/n})$ for all $(q_v, n) = 1$, which shows the equality.
\end{proof}
This lemma naturally leads to the following groupoid.
\begin{definition}
	Assume that $G$ is unramified and tame at $v$, we may thus think of it as a $\Gamma_{\F_v}$-group. Define the groupoid $[G(-1)/G](\F_v)$ as follows.
	\begin{itemize}
		\item It has as objects pairs $(\gamma, F)$ where $\gamma \in G(-1)$, $F \in G$ and $\gamma = F \cdot \Frob_v(\gamma) \cdot F^{-1}$. 
		\item Morphisms are given by conjugation with elements of $G$. More concretely, for every element $g \in G$ we get a morphism $(\gamma, F) \to (g^{-1} \gamma g, g^{-1} F \Frob_v(g))$.
	\end{itemize}
\end{definition}

This groupoid is a quotient stack with respect to the conjugation action of $G$ on $G(-1)$.

\begin{definition}
	Fix a uniformizer $\pi_v$ and for each $(q_v, n)$ let $\pi_v^{1/n}$ be a compatible system of roots. Let $F \in \Gamma_{k_v}^{\mathrm{tame}}$ be the lift of Frobenius constructed in Lemma \ref{lem:lift_of_frobenius}. We define the reduction modulo $\pi_v$ map as the map of groupoids
	\[
	\pmod{\pi_v}: BG(k_v) \to [G(-1)/G](\F_v), \quad  \varphi \to (\varphi \circ \theta, \varphi(F)).
	\]
\end{definition}

This is a categorification of the ramification type in the following sense: There is a commutative diagram
\[
\xymatrix{BG(k_v)  \ar[rr]^{\bmod{\pi_v}} \ar[drr]^{\rho_{G,v}} & & [G(-1)/G](\F_v) \ar[d] \\ 
& & \mathcal{C}_G^{\Gamma_{k_v}}
}
\]
where the right hand map is $(\gamma,F) \mapsto \gamma$. To see that this map is well-defined, we note that the relation $\gamma = F \cdot \Frob_v(\gamma) \cdot F^{-1}$ implies that $\gamma$ is conjugate to its Galois conjugates

Our main result on this from \cite[\S4.6]{LS24} is the following, which we view as a stacky version of Hensel's Lemma.

\begin{theorem}[Stacky Hensel's Lemma]\label{thm:hensel}
	The map $BG(k_v) \to [G(-1)/G](\F_v)$ is well-defined and an equivalence of groupoids.
\end{theorem}
\begin{proof}
	To see that it is well-defined one applies the cocycle law for $\varphi$ to the equality $\theta(\Frob_v(\cdot)) = F \theta(\cdot) F^{-1}$ of Lemma \ref{lem:lift_of_frobenius}.
	
	For the equivalence of groupoids we refer to \cite[Thm.~4.13]{LS24} and \cite[Lem.~4.12]{LS24} to see that this theorem is about this map. 
\end{proof}

\subsection{Weight functions}

\begin{definition} \label{def:weight_function}
A \textit{weight function} (over $k$) is a function $w: \mathcal{C}_G \to \Z$ such that 
\begin{enumerate}
	\item $w$ is invariant under the action of $\Gamma_k$.
	\item $w(e) = 0$,	
\end{enumerate} 
Here $e$ denotes the identity map of $G(-1)$. 
\end{definition}

Note that weight functions form a finitely generated free abelian group under addition. In analogy with Manin's conjecture, weight functions play the role of line bundles.

In \cite[\S8.1]{LS24} we use a slightly more general framework for heights based on \textit{orbifold line bundles}. In this survey we take a slightly simplified perspective using weight functions only, which is sufficient for almost all applications, since they form a finite index subgroup of the full orbifold Picard group.

\begin{example}
	If $w$ is invariant under the invertible powering action, 
	i.e.~$w(c) = w(c^q)$ for all $q \nmid |G|$ and all $c \in \mathcal{C}_G$,
	then we can view $w$ as a function 
	on the anticyclotomic twist of $G$ (see Lemma \ref{lem:Galois_action_on_G(-1)}).
	This allows for easier computations and easier to write down $w$ via elements of $G$,
	instead of $G(-1)$.
	
	For $G$ with trivial Galois action over $\Q$, the Galois invariance condition says exactly that $w$ is invariant under the invertible powering action. 
	But over other number fields the condition can be more general.
\end{example}

\subsection{Height functions}
With heights on varieties there are many heights with given line bundle. We also allow many heights with given weight function. This additional flexibility is very useful for the theory, particularly when we come to study equidistribution.   Not least because there is no canonical choice of local height at the wild, archimedean, or other bad places. We let $G$ be a $\Gamma_k$-group over a number field $k$.

\begin{definition} \label{def:heights}
A \textit{local height function} at a place $v$ is a map 
$$H_v: BG(k_v) \to \R_{> 0},$$
which is constant on isomorphism classes of objects. An \emph{adelic height} associated to a weight function $w$ is a collection $H=(H_v)_v$ of local heights for all places $v$ of $k$ such that for all but finitely many good places $v$ we have
\begin{equation} \label{eqn:tame_height}
H_v(\varphi_v) = q_v^{w(\rho_{G,v}(\varphi_v))} \quad \text{ for all } \varphi_v \in BG(k_v),
\end{equation}
where $\rho_{G,v}$ denotes the ramification type at $v$ from \S \ref{sec:ramification_type}. Inspired by the formula \eqref{def:naive_height}, the (global) height of $\varphi \in BG(k)$ is then defined to be the product
\begin{equation} \label{def:product_local_heights}
H: BG(k) \to \R_{>0}, \quad \varphi \mapsto \prod_v H_v(\varphi)
\end{equation}
over all local heights, where we consider the image of $\varphi$ under the map $BG(k) \to \prod_v BG(k_v)$. For simplicity we often call an adelic height simply a height. 
\end{definition}

\begin{remark}
	If $G$ has trivial Galois action, then $BG[k] \to \prod_v BG[k_v]$ is injective since any Galois
	extension is uniquely determined by its completely split primes. However it fails to be injective
	in general, for example if $\Sha^1(k,G)$ is non-trivial.
\end{remark}

\begin{example}[Artin conductors]\label{ex:Artin_conductors}
	Let $G$ be a finite group acting on a finite dimensional vector space $V$ over $\C$.
	Then the Artin conductor associated to this is a height function with associated weight function
	$$w:G \mapsto \Z, \quad g \mapsto \mathrm{codim}(V^g).$$
	To see this, the $v$-adic valuation of the Artin conductor at a tame place $v$ is exactly 
	$\chi(1) - \chi(I)$, where $\chi(I)$ denotes the average value of the associated 
	character $\chi$ over all elements of the inertia group $I$. By character orthogonality this is $\dim(V) - \dim(V^I)$. The claim follows as $I$ is cyclic  by Lemma \ref{lem:tame_inertia}.
\end{example}

\begin{example}[Discriminants] \label{ex:discriminant}
	Let $G \subset S_n$ be a transitive subgroup. Recall from Lemma \ref{lem:S_n} that a general
	element of $BG(k)$ corresponds to a number field $K/k$ of degree $n$ whose Galois closure 
	has Galois group $G$. We explain how the discriminant $\Delta_{K/k}$ can be viewed
	as a height function on $BG(k)$ with associated weight function the \emph{index:}
	$$\ind: S_n \to \R, \quad g \mapsto n - \text{ the number of orbits of $g$ on }\{1,\dots,n\}.$$
	It suffices to realise the discriminant as a conductor via Example \ref{ex:Artin_conductors}. We claim that the discriminant is equal to the conductor of the permutation representation $V$ induced by the embedding of $G \subset S_n$. For $g \in G$ the codimension $V^g$ is equal to $n$ minus the number of cycles of $g$ as a permutation; this is exactly $\ind(g)$.
	
	We provide a careful argument for the equality of the discriminant and this Artin conductor, using a suitable version of the conductor-discriminant formula, as we were unable to locate one in the literature. Note that particular care is required since Example \ref{ex:Artin_conductors} is for Galois extensions but we are considering non-Galois extensions.
	
	The degree $n$ extension corresponding to a Galois $G$-extension for a transitive subgroup $G \subset S_n$ corresponds under the Galois correspondence to the subgroup $H := G \cap S_{n-1} \subset G$. The compatibility of induced representations with Artin conductors in \cite[Cor.~VII.11.8]{Neu99} then implies that the discriminant of this extension is given by the Artin conductor of the induced representation $\mathrm{Ind}_{H}^G \mathbf{1}$, where $\mathbf{1}$ denote the trivial irreducible representation. This is by definition the same as the permutation representation coming from $G \subset S_n$. This implies the claim.
\end{example}

Naturally with Definition \ref{def:heights}, there is no guarantee that there are only finitely many elements of bounded height (the Northcott property). To ensure this, one usually considers \emph{big heights}.

\begin{definition}
	A height is called \emph{big} if the associated weight function takes positive values on all non-identity elements.
\end{definition}

\begin{lemma}
	Every big height satisfies the Northcott property.
\end{lemma}
\begin{proof}
	First assume that the Galois action on $G$ is trivial. Let $\Delta$ be the height function given by the discriminant of the corresponding field. Then there exists $\varepsilon > 0$ such that $H(\varphi) > \Delta(\varphi)^{\varepsilon}$ for all $\varphi \in BG(k)$. Hence bounding the height bounds the discriminant. Thus Hermite--Minkowski implies there are finitely many of bounded height.
	
	For general $G$, given $H$ and $B> 0$ there is a finite set of places $S$ such if $H(\varphi) < B$ then $\varphi \in BG[\O_v]$ for all $v \notin S$, i.e. the cocycle is unramified outside of $S$. 
	However there are only finitely many such cohomology classes by \cite[Prop.~5.1]{GMB13}.
\end{proof}

Recall from \eqref{def:product_local_heights} that the height is given as a product of local heights. The counting function is naturally controlled by cocycles of small height. To make the height small we need to make the $v$-adic valuations of the local heights small; this means exactly taking the ramification type to have minimal weight. Therefore the minimal weight conjugacy classes play a crucial role to understanding the size of the height function and the distribution of cocycles of bounded height; they will frequently appear in our work.

\begin{definition} \label{def:minimal_weight}
	Let $H$ be a height function with weight function $w$. 
	The \emph{minimal weight conjugacy classes} of $w$ is the collection
	$$\mathcal{M}(H):=\{ c \in \mathcal{C}_G^* : w(c) \text{ is minimal} \}.$$
\end{definition}

\subsection{Balanced heights} \label{sec:balanced}
Wood \cite{Woo10} was the first to notice that the discriminant can sometimes exhibit pathological properties. For example when ordering cyclic quartic fields by discriminant, a positive proportion of all cyclic quartic fields contain a given quadratic subfield (providing there is at least one such field). This wreaks havoc with the leading constant and it loses many formal properties one would like.

To rectify this, Wood introduced the notion of \textit{fair heights} in the case of finite abelian groups. These heights have the property that when counting, the leading constant is given by a finite sum of Euler products \cite[Thm.~3.1]{Woo10}.

It has been known for a long time in the Manin's conjecture world what the ``nice'' class of heights ought to be. As explained in \S \ref{sec:leading_constant_Manin} these are of those associated to an adjoint rigid line bundle. In \cite{LS24} we applied this philosophy to the stack $BG$ and obtained a class of height functions which is closely related to Wood's fair heights, but slightly more general and can also be defined in the non-abelian case. For such heights the expected leading constant has good measure-theoretic properties, provided one removes a thin set of possible bad field extensions, and can also be written as a finite sum of Euler products as in Wood's paper (see Conjecture \ref{conj:balanced}).

\begin{definition} \label{def:balanced}
	A height function $H$ is called \emph{balanced} if the collection of minimal weight conjugacy classes $\mathcal{M}(H)$ generates $G$.
\end{definition}

We use the terminology \emph{balanced} as it softer than adjoint rigid, and also seems better suited and more descriptive in the setting of Malle's conjecture.


\begin{example}[Anticanonical heights] \label{ex:radical_discriminant}
	We define the \emph{anticanonical weight function} to be $w(c) = 1$ for all non-identity $c \in \mathcal{C}_G$.
	An \emph{anticanonical height function} is any height function associated to this weight function.
	For example the radical discriminant
	$$\rad \Delta_K := \prod_{\mathfrak{p} \mid \Delta_K} \Norm(\mathfrak{p})$$
	is an anticanonical height function. In many ways this is ``the best'' height function as it weights all field extensions in the fairest and most equal way.
	
	To justify the terminology from an arithmetic perspective, we remark in Manin's conjecture we saw in \S \ref{sec:anticanonical} that anticanonical height functions are those which give the largest predicted power of $\log B$. The same is true here (see Conjecture \ref{conj:balanced}).
\end{example}

\begin{example}[The discriminant]
	Let $G \subset S_n$ be a transitive subgroup. Consider the discriminant height $H$, as in Example \ref{ex:discriminant}. Then $\mathcal{M}(H)$ consists of the conjugacy classes of minimal index. For $G=S_n$ these are exactly the transpositions, thus the discriminant is balanced for $S_n$ since the transpositions generate $S_n$.
	
	But if $G \subset S_{|G|}$ is abelian, then the discriminant is balanced if and only if $G = C_p^m$ with $p$ prime. In fact it seems that the discriminant is very rarely balanced as one varies over finite transitive permutation groups.
\end{example}

\section{Tamagawa measures} \label{sec:Tamagawa}
In arithmetic geometry, Tamagawa measures are measures defined on the adelic points $X(\Adele_k)$ of some variety $X$ over a number field $k$. One first defines local Tamagawa measures at each place $v$, then to ensure the convergence of the product one needs to come up with \emph{convergence factors}, which typically come from the special values of some local $\L$-functions. Weil introduced these for algebraic groups \cite{Wei82}, leading to his famous Tamagawa number conjecture. They also appear for abelian varieties in the Birch and Swinnerton-Dyer conjecture.

One of Peyre's key observations \cite[\S 2.2.1]{Pey95} was that to define Tamagawa measures on Fano varieties, one needs to choose a height function first (or more specifically, an \textit{adelically metrised line bundle}). The difference with algebraic groups is that the tangent bundle is canonically isomorphic to the trivial vector bundle which leads to a canonical choice of metric, but on Fano varieties there is no such canonical choice. From a more down-to-earth perspective, it is clear that the leading constant in the asymptotic formula should depend on the choice of height function, and one would like an adelic interpretation of this leading constant. Chambert-Loir and Tschinkel subsequently \cite{CT10} came up with an approach which unifies both Peyre's and Weil's.

\subsection{Topologies on adelic points} \label{sec:topology}
Before we define a measure, we need to make clear which topological space we are working on. It turns out that the choice of topological space to take depends on the choice of height function. This perspective  also appears in Manin's conjecture \eqref{eqn:adjoint_rigid}, where one takes the adelic points on an open subvariety.

Firstly $BG[k_v]$ is finite so we equip it with the discrete topology. There are various choices for restricted direct products of these sets, given in the following definition.

\begin{definition} \label{def:adelic_space}
	Let $\mathcal{C}\subseteq \mathcal{C}_G^*$ be Galois invariant.
	Then for a good place $v$  we define 
	\begin{align*}
	BG(\O_v)_{\mathcal{C}} &= \{ \varphi_v \in BG(k_v) : \rho_{G,v}(\varphi_v) \in \mathcal{C} \cup \{e\}\}, \\
	BG(\Adele_k)_{\mathcal{C}} & = \lim_{\xrightarrow[S]{}} \prod_{v \in S} BG(k_v) 
	\prod_{v \notin S} BG(\O_v)_{\mathcal{C}},
	\end{align*}
	where the limit is over all finite sets of places $S$ containing the bad places, so that $BG(\Adele_k)_{\mathcal{C}}$
	is the restricted direct product of the $BG(k_v)$ with respect to 	$BG(\O_v)_{\mathcal{C}}$.
	We call this the \emph{partial adelic space} with respect to $\mathcal{C}$
	and  an element of $BG(\O_v)_{\mathcal{C}}$ a \textit{partial $v$-adic integral point} with respect to 
	$\mathcal{C}$.
\end{definition}

An element of $BG(\O_v)_{\mathcal{C}}$ exactly corresponds to a cocycle whose ramification type is restricted to lie in $\mathcal{C}$. It follows easily from Definition \ref{def:unramified} that $BG(\O_v) = BG(\O_v)_{\emptyset}$, since a cocycle is unramified if and only if its ramification type is trivial.

The space $BG(\Adele_k)_{\mathcal{C}}$ is locally compact. It is compact if and only if
$\mathcal{C} = \mathcal{C}_G^*$. 

\subsection{Local Tamagawa measures}
Fix an adelic height $(H_v)_v$ with corresponding weight function $w$ and write $a(H) := \left(\min_{c \in \mathcal{C}_G^*}w(c)\right)^{-1}$.

We now apply Peyre's formalism \cite{Pey95} for Tamagawa measures to the stack $BG$. Thankfully the formulae which arise exactly agree with the existing notions which appear in the Malle's conjecture literature; for example Kedlaya's ``total mass'' from \cite[Def~2.2, (2.3.1)]{Ked07}. This leads to the following definition.

\begin{definition} \label{def:local_Tamagawa_measure}
	Let $v$ be a place of $k$ and $W_v \subseteq BG[k_v]$. We define the \textit{Tamagawa measure}
	$\tau_{H,v}$ associated to our choice of adelic height to be
	$$\tau_{H,v}(W_v) = \sum_{\varphi_v \in [W_v]} \frac{1}{|\Aut(\varphi_v)| H_v(\varphi_v)^{a(H)}},$$
	i.e.~the integral of $H_v(\varphi_v)^{-a(H)}$ with respect to the groupoid cardinality. 
\end{definition}
This sum is finite as there are only finitely many extensions of $k_v$ of given degree. This determines a well-defined measure on the set $BG[k_v]$ of isomorphism classes of $k_v$-points of $BG$.

Lemma \ref{lem:groupoid_count} allows one to rewrite this in terms of a count over $1$-cocycles with a different weighting. Namely if $F: Z^1(k_v,G) \to BG(k_v)$ denotes the natural map then
\begin{equation} \label{eqn:tau_alternative}
	\tau_{H,v}(W_v) = \frac{1}{|G|}\sum_{\varphi_v \in F^{-1}(W_v)} \frac{1}{H_v(F(\varphi_v))^{a(H)}}.
\end{equation}

There is a formula for the local Tamagawa measure of $BG[k_v]$ away from finitely many places, as well as the more general partial local points considered in Definition~\ref{def:adelic_space}. Having such a formula is crucial to show convergence of the global Tamagawa measure. Inspired by Bhargava's paper \cite{Bha07}, we call such a formula a \emph{mass formula}. We call a place \emph{good} with respect to $(G,H)$ if it is good with respect to $G$ and if the formula \eqref{eqn:tame_height} holds at $v$.

\begin{theorem}[Mass formula] \label{thm:mass_formula}
	Let $\mathcal{C} \subset \mathcal{C}_G^*$ be a Galois invariant subset. Then for good places $v$ we have
 	\begin{align*}
		\tau_{H, v}(BG(k_v))= \sum_{c \in \mathcal{C}_G^{\Gamma_{k_v}}} q_v^{-w(c)a(H)}, \quad
		\tau_{H, v}(BG(\O_{v})_{\mathcal{C}}) = 1 +
		\sum_{c \in \mathcal{C}^{\Gamma_{k_v}}} q_v^{-w(c)a(H)}.		
	\end{align*}
\end{theorem}
\begin{proof}
	The first equality is a special case of the second one for $\mathcal{C} = \mathcal{C}_G^*$.
	
	Given that the height is defined via the ramification type, the natural approach is to sort elements of $BG(k_v)$
	according to their ramification type. This yields
	$$\tau_{H, v}(BG(\O_{v})_{\mathcal{C}}) = \sum_{c \in \mathcal{C}^{\Gamma_{k_v}}} q_v^{-w(c)a(H)} \sum_{\substack{\varphi_v \in BG[k_v] \\ \rho_{G,v}(\varphi_v) = c}} \frac{1}{|\Aut \varphi_v|}.$$
	Thus it suffices to show that 
	\begin{equation} \label{eqn:groupoid_cardinality_1}
		\sum_{\substack{\varphi_v \in BG[k_v] \\ \rho_{G,v}(\varphi_v) = c}} \frac{1}{|\Aut \varphi_v|} = 1
		\quad \text{ for all }c \in \mathcal{C}_G^{\Gamma_{k_v}}.
	\end{equation}
	The proof of this consists of two parts: firstly showing that the ramification type
	$\rho_{G,v} : BG[k_v] \to \mathcal{C}_G^{\Gamma_{k_v}}$ is surjective, and secondly calculating
	the fibres of the map using that it is non-empty.
	
	We explain these steps under the assumption that $G$ is abelian. This requires a cohomological interpretation
	of the ramification type. Inflation-restriction \cite[\S 2.6(b)]{Ser02} applied to the subgroup
	$I_v \subset \Gamma_{k_v}$ yields the exact sequence
	$$0 \to \H^1_{\textrm{nr}}(k_v, G) \to \H^1(k_v, G) \to \H^1(I_v, G)^{\Gamma_{k_v}} \to \H^2(\Gal(k_v^{\textrm{nr}}/k_v),G).$$
	Here we use that as $G$ is unramified at $v$, the inertia group $I_v$ acts trivially on $G$. 
	It follows easily from Definition \ref{def:ramification_type} that the map 
	$\H^1(k_v, G) \to \H^1(I_v, G)^{\Gamma_{k_v}}$ is exactly the ramification type; in particular we see that
	$\rho_{G,v}$ is a group homomorphism. However we have $\Gal(k_v^{\textrm{nr}}/k_v) \cong \dual{\Z}$, 
	which has cohomological dimension $1$ \cite[\S3.2]{Ser02}, so $\H^2(\Gal(k_v^{\textrm{nr}}/k_v),G) = 0$. In particular
	we deduce that $\rho_{G,v}$ is surjective. As $\rho_{G,v}$ is a group homomorphism the fibres all have the same
	cardinality, namely $\#\H^1_{\textrm{nr}}(k_v, G)$. But it follows from Lemma \ref{lem:unramified_cocycle} that 
	$\#Z^1_{\textrm{nr}}(k_v, G) = |G|$, whence \eqref{eqn:groupoid_cardinality_1} immediately follows from \eqref{eqn:tau_alternative}.
	
	If $G$ is no longer abelian, then every step above breaks. The ramification type
	$\rho_{G,v}$ is no longer a group homomorphism because there is no group structure on non-abelian $\H^1$.
	Moreover there is no version of the inflation-restriction exact sequence for non-abelian group cohomology
	where one encounters non-abelian $\H^2$. However one of the key applications of algebraic stacks is a version of
	non-abelian $\H^2$ via gerbes, as worked out by Giraud \cite{Gir71}. So to prove the result in the non-abelian
	case in \cite[Lem.~8.8, Thm.~8.10]{LS24} we instead give a proof using gerbes over finite fields;
	this proof is philosophically a non-abelian version of the above proof and using some of the techniques and results from \S \ref{sec:categorification}.
\end{proof}

In particular, for the minimal weight conjugacy classes we obtain simply
\begin{equation} \label{eqn:measure_minimal_weight}
\tau_{H, v}(BG(\O_{v})_{\mathcal{M}(H)}) = 	1+ \frac{\#\mathcal{M}(H)^{\Gamma_{k_v}}}{q_v}.
\end{equation}

\begin{example} \hfill

	(1) Our mass formula is a formula for the Tamagawa measure of $BG(k_v)$.
	This terminology is inspired by Bhargava's mass formula for $S_n$ \cite[Thm.~1.1]{Bha07}.
	We explain how to recover Bhargava's formula, though we emphasise that
	Bhargava's formula applies to any place, but we are only able to recover it at the
	 \emph{tame} places.
	Finding a stack-theoretic generalisation of Bhargava's formula at the wild places
	looks like an interesting and challenging problem.
	
	So let $v$ be a tame place for $S_n$. The Galois action on $\mathcal{C}_{S_n}$ is trivial
	as $S_n$ has rational character table (see Example \ref{ex:Galois_action}). Thus the right hand
	side of our mass formula reads
	$$ \sum_{ c \in S_n/\mathrm{conj}} q_v^{-\ind(c)} = 
	\sum_{k  = 0}^{n-1} q_v^{-k} \#\{ c \in S_n/\mathrm{conj}: \ind(c) = k\}.$$
	The conjugacy classes of $S_n$ are given by cycle types.
	So the number of conjugacy classes of $S_n$ of given index $k$ is exactly equal
	$q(k,n-k) :=$ the number of partitions of $k$ into at most $n-k$ parts.
	This recovers \cite[Thm.~1.1]{Bha07}.
	
	(2) We now consider $G = A_4$ with the weight function given by the index, 
	as in Example \ref{ex:discriminant}.
	The non-identity conjugacy classes of $A_4$ are as follows, with the following indices:
	$$
	\begin{tabular}{c|cccc}
		c& (1,2)(3,4) & (1,2,3) & (1,3,2) \\ \hline
		$\ind(c)$ & 2 & 2 & 2 
	\end{tabular}$$
	As $(1,2,3)^5 = (1,3,2)$, we find that these two conjugacy classes
	are swapped by the anticyclotomic action on $\mathcal{C}_{A_4}$, which factors
	through $\Q(\zeta_3)$. Thus for $p > 3$,
	Theorem \ref{thm:mass_formula} yields
	$$\tau_{p}(BG(\Q_p)) =
	\begin{cases}
		 1 + \frac{3}{p}, & \text{ if }p \equiv 1 \bmod 3, \\
		 1 + \frac{1}{p}, & \text{ if }p \equiv 2 \bmod 3. \\		 
	\end{cases}$$	
	This calculation was used to put forward a precise prediction for the leading constant when counting $A_4$-quartics of bounded discriminant in \cite[Conj.~1.1]{LS24}.
\end{example}

\begin{remark} \label{rem:surjective_ramification_type}
	One consequence of Theorem \ref{thm:mass_formula} is that the ramification type map
	$$\rho_{G,v}: BG[k_v] \to \mathcal{C}_G^{\Gamma_{k_v}}$$
	is surjective at the good places $v$; this is not at all obvious from the definition of $\rho_{G,v}$.
	Indeed as the proof of Theorem \ref{thm:mass_formula} illustrates, this requires the vanishing of some non-abelian $\H^2$.
	We utilise this in \S\ref{sec:heuristic} to develop a heuristic for Malle's conjecture.
\end{remark}

\subsection{Global Tamagawa measures} \label{sec:global_Tamagawa}
We now wish to take the product of local measures. To do so we require convergence factors. 

For any $\Gamma_k$-subset $\mathcal{R}$ we can consider the associated permutation representation $\C[\mathcal{R}]$, given by the finite-dimensional vector space with basis $\mathcal{R}$ and $\Gamma_k$ acting by permuting the basis.

We denote by $\L(\mathcal{R},s) := \L(\C[\mathcal{R}],s)$ the corresponding Artin $\L$-function with local Euler factors $\L_v(\mathcal{R},s)$, where we take  $\L_v(\mathcal{R},s) = 1$ for $v$-archimedean. The Artin $\L$-function $\L(\mathcal{R},s)$ is actually relatively easy to describe, namely it is a product of the Dedekind zeta functions.

\begin{lemma} \label{lem:Dedekind_zeta}
	Let $\mathcal{R} \subset \mathcal{C}_G^*$ be a Galois invariant subset. Then
	$$\L(\mathcal{R},s) = \prod_{c \in \mathcal{R}/\Gamma_k} \zeta_{k(c)}(s)$$
	where $k(c)$ denotes the residue field of $c$, as in Definition \ref{def:residue_field}.
\end{lemma}
\begin{proof}
	The representation $\C[\mathcal{R}]$ is a direct sum of representations induced from the trivial representation with respect to the field extensions $k \subset k(c)$. The result then follows from the Artin formalism for $\L$-functions.
\end{proof}

Recall the minimal weight conjugacy classes from Definition \ref{def:minimal_weight}. Let $\L^*(\mathcal{M}(H),1) := \lim_{s \to 1}(s-1)^{b(k,H)} \L(\mathcal{M}(H),s)$,
where $b(k,H):= \#\mathcal{M}(H)/\Gamma_k$. This is non-zero as $b(k,H)$ is the order of pole at $s= 1$ by Lemma \ref{lem:Dedekind_zeta}. For each place $v$ of $k$ we define
$$ \lambda_v = \L_v(\mathcal{M}(H),1).$$
We emphasise that our convergence factors depend on $H$.
\begin{definition} \label{def:Tamagawa}
The  global Tamagawa measure associated to $H$ is 
\begin{equation} 
	\tau_H = \L^*(\mathcal{M}(H),1)\prod_v \lambda_v^{-1} \tau_{H,v}.
\end{equation}
\end{definition}

It follows relatively quickly from the mass formula (Theorem \ref{thm:mass_formula}) that the product defining $\tau_H$ is absolutely convergent on both $\prod_v BG[k_v]$ and $BG(\Adele_k)_{\mathcal{M}(H)}$ (see \cite[Thm.~8.17]{LS24}). We study now an important class of measurable sets.

\begin{definition} \label{def:basic_open}
	Let $\mathcal{C} \subset \mathcal{C}_{G}$ be Galois invariant.
    A \textit{large open} of $BG(\Adele_k)_{\mathcal{C}}$
    is a subset of the form $W = \lim_S \prod_{v \in S} W_v \prod_{v \not \in S} BG(\O_v)_{\mathcal{C}}$ with $W_v \subseteq BG[k_v]$ such that
    $$BG(\O_v)_{\mathcal{C}} \subseteq W_v$$
    for all but finitely many good $v$. 
    A \textit{basic open} is a large open such that
    $W_v = BG(\O_v)_{\mathcal{C}}$ for all but finitely many $v$.
\end{definition}

\begin{lemma} \label{lem:large_open}
	Any large open of $BG(\Adele_k)_{\mathcal{M}(H)}$ is open, closed, and has positive Tamagawa measure provided
	it is non-empty.
\end{lemma}
\begin{proof}
	It is open and closed because the same holds for all suitably large finite sets of places $S$
	when restricted to 
	$\prod_{v \in S} BG(k_v) \prod_{v \notin S} BG(\O_v)_{\mathcal{M}(H)}$. It suffices to show positive measure for basic opens. However this is immediate
	from \eqref{eqn:measure_minimal_weight}.
\end{proof}

\begin{example} \label{ex:restricted_ramification_type}
    Let $\mathcal{C} \subset \mathcal{R} \subset \mathcal{C}_G^*$ be Galois invariant and let $S$ be a finite set of places. Then 
    $$ \{ (\varphi_v) \in BG(\Adele_k)_{\mathcal{C}} : \varphi_v \in BG(\O_{v})_{\mathcal{R}} \text{ for all } v \notin S\}$$
    is a large open; it consists of those adelic points whose ramification type is restricted to lie in $\mathcal{R}$ for all $v \notin S$.
\end{example}

There are alternative convergence factors which are easier to write down, but care is required when working with them since they only give a conditionally convergent Euler product in general. See \cite[Lem.~8.19]{LS24} for the explanation why the following definition is equivalent to Definition \ref{def:Tamagawa}.


\begin{definition}[Naive convergence factors]
Write $\zeta_k^*(1) =  \lim_{s \to 1}(s-1) \zeta_k(s)$. Then we have
\[
\tau_H = \zeta_k^*(1)^{b(k,H)}\prod_v (1 - q_v^{-1})^{b(k,H)} \tau_{H,v}
\]
where we set $q_v^{-1} =0$ for $v$ archimedean.
\end{definition}


\section{Brauer groups} \label{sec:Brauer}
We now come to the main new innovation of the paper \cite{LS24}, which is the use of Brauer groups to explain the leading constant in Malle's conjecture. To make the exposition as simple as possible, we avoid all stacky terminology (aside from the definition of $BG$) and use the explicit description of Brauer group elements on $BG$ obtained in \cite[\S6.2]{LS24} via \emph{central extensions}.

\subsection{Pushouts}
We recall the following standard construction from category theory which we will use multiple times. Let $A \to B$ and $A \to C$ be morphisms of groups. Then the \emph{pushout} is the group $D$ which is universal with respect to the diagram
\[
\xymatrix{ A \ar[r] \ar[d] & B \ar[d] \\ 
C  \ar[r] & D. }
\]
If $A,B,C$ are $\Gamma_k$-groups, then $D$ admits a structure of a $\Gamma_k$-group, which is also unique.

\subsection{Central extensions}
Let $k$ be a field. We will consider the $\Gamma_k$-module of roots of unity in $k^{\sep}$.
\[
\mu_\infty := \{\zeta \in k^{\sep}: \exists n \in \N: \zeta^n = 1\} = \bigcup_{p \nmid n} \mu_n. 
\]

\begin{definition} \label{def:central}
	Let $G$ be finite $\Gamma_k$-group. A \emph{Brauer element} is a central extension of $\Gamma_k$-groups
	\begin{equation} \label{seq:central}
	\beta: 1 \to \mu_{\infty} \to G_{\beta} \to G \to 1.
	\end{equation}
\end{definition}

We emphasise in this definition: even if $G$ has trivial Galois action, the group $G_{\beta}$ need not have trivial Galois action.

 	The set of such extensions admits a commutative group structure given by the Baer sum. Given central extensions $\beta, \beta'$ the Baer sum is defined by the following pushout diagram
 	\[\begin{tikzcd}
 		1 & {\mu_{\infty} \times \mu_{\infty}} & {G_{\beta} \times_G G_{\beta'}} & G & 1 \\
 		1 & {\mu_{\infty}} & {G_{\beta + \beta'}} & G & 1 
 		\arrow[from=1-1, to=1-2]
 		\arrow[from=1-2, to=1-3]
 		\arrow[from=1-2, to=2-2, "(\zeta{,} \zeta{'}) \to \zeta \cdot \zeta{'}"]
 		\arrow[from=1-3, to=1-4]
 		\arrow[from=1-3, to=2-3]
 		\arrow[from=1-4, to=1-5]
 		\arrow[from=1-4, to=2-4, equal]
 		\arrow[from=2-1, to=2-2]
 		\arrow[from=2-2, to=2-3]
 		\arrow["\lrcorner"{anchor=center, pos=0.125, rotate=180}, draw=none, from=2-3, to=1-2]
 		\arrow[from=2-3, to=2-4]
 		\arrow[from=2-4, to=2-5]
 	\end{tikzcd}\]

	The inverse is given by applying inversion to the map $\mu_\infty \to G_\beta$, and the identity given by the trivial (or split) central extension  $G_0:= \mu_\infty \times G$ with the natural maps. We consider central extensions up to the usual notion of equivalence \cite[\href{https://stacks.math.columbia.edu/tag/010J}{Tag 010J}]{stacks-project}.
	
\begin{lemma}
	The collection of central extensions \eqref{seq:central} up to equivalence defines an abelian group under Baer sum. We call this the (normalised) Brauer group of $BG$, and denote it by $\Bre BG$.
\end{lemma}
\begin{proof}
	To check that the Baer sum gives a well-defined group structure one has to check that the isomorphisms witnessing the group laws for the Baer sum preserve the Galois action. This follows from the standard functoriality properties of the Baer sum \cite[\href{https://stacks.math.columbia.edu/tag/010L}{Tag 010L}]{stacks-project}.
\end{proof}

\begin{remark}
	We compare the definition with the usual $\Br BG := \H^2(BG,\Gm)$ as used in \cite{LS24}. It is shown in \cite[Lem.~6.7]{LS24} that any element of $\Br BG$ which evaluates trivially at the identity cocycle of $BG$ may be represented by a central extension, as in Definition \ref{def:central}. Thus in our survey we take this to be the \emph{definition} of the Brauer group of $BG$. We use slightly different notation to take into account this renormalisation; to compare between the two papers we have $\Br BG = \Br k \oplus \Bre BG$. The constant elements $\Br k$ will play no role in this survey, essentially because they play no role in the Brauer--Manin obstruction.
\end{remark}

For explicit computations, it is often easier to work with \textit{finite} $\Gamma_k$-groups. Given a central extension
\begin{equation} \label{def:central_extension_mu_n}
	\beta: 1 \to \mu_n \to G_{\beta} \to G \to 1
\end{equation}
with kernel $\mu_n$ for some $n$, one obtains a central extension 
\begin{equation} \label{eqn:push_out_mu_infty}
\begin{split}
\xymatrix{ 1 \ar[r] & \mu_n \ar[r] \ar[d] & G_{\beta} \ar[r] \ar[d] & G \ar[r] \ar[d] &  1 \\ 
1 \ar[r]& \mu_\infty \ar[r] & G_{\infty,\beta}   \ar[r] & G \ar[r] & 1. }
\end{split}
\end{equation}
with kernel $\mu_\infty$ by applying a pushout, thus a Brauer group element. The disadvantage is that the choice of representation need not be unique. This discrepancy is recorded in the following lemma.

\begin{lemma} \label{lem:trivial_in_Brauer_group}
	Let 
	$$	\beta: 1 \to \mu_n \to G_{\beta} \to G \to 1$$
	be a central extension. Then $\beta$ represents the trivial class in the Brauer group 
	if and only if there exists a homomorphism $G_{\beta} \to \mu_\infty$ which is the identity on $\mu_n \subset G_\beta$.
\end{lemma}
\begin{proof}
	We use the diagram \eqref{eqn:push_out_mu_infty}.
	Assume that the corresponding Brauer group element is trivial. Then $G_{\infty,\beta} \cong \mu_\infty \times G$. The map $G_\beta \to \mu_\infty$ is just given by composition with the 
	projection map. Conversely
	let $\chi : G_\beta \to \mu_\infty$ be as in the statement. Then the universal property of pushout
	yields a map $G_{\infty,\beta} \to \mu_\infty$ which is a retraction of $\mu_\infty \to G_{\infty,\beta}$. Thus the central extension splits, hence is a product. 
\end{proof}

A more explicit description of those central extensions which represent the trivial class can be found in \cite[Lem~6.6]{LS24}.

The following shows which values of $n$ one can work with (for example one can take $n = |G|^2$).

\begin{lemma} \label{lem:exp}
	Any Brauer group element may be represented by a central extension with kernel $\mu_{n}$ for $n = \exp(G) \cdot \exp(\H^2(G, \Q/\Z))$, where $\exp(G)$ denotes the exponent of $G$ and $\exp(\HH^2(G, \Q/\Z))$ is the exponent of the Schur multiplier.
\end{lemma}
\begin{proof}[Proof sketch]
	Let $(G_{\beta}, \beta)$ be a Brauer element. We disregard the Galois actions for now. Recall that central extensions of $G$ by $A$ are classified by the group cohomology $\HH^2(G, A)$  \cite[Thm.~IV.3.12]{Bro94}. For all $m$ we have the Kummer exact sequence
	\[
	\HH^1(G, \mu_{\infty}) \xrightarrow{\cdot m} \HH^1(G, \mu_{\infty}) \to \HH^2(G, \mu_{m}) \to \HH^2(G, \mu_{\infty}) \xrightarrow{\cdot m} \HH^2(G, \mu_{\infty})
	\]
	We have $\HH^2(G, \mu_{\infty}) = \HH^2(G, \Q/\Z)$ as abstract groups. Taking $m =\exp(\H^2(G, \Q/\Z))$ we may thus find a central extension of abstract groups 
	\[
	\alpha: 1 \to \mu_{\exp(\H^2(G, \Q/\Z))} \to \tilde{G} \to G \to 1
	\]
	whose pushout to $\mu_{\infty}$ is isomorphic to $(G_{\beta}, \beta)$ as an extension of abstract groups.
	
	Handling the  Galois action is more delicate as it requires considering obstruction classes in cohomology which leads to a kernel of size $\exp(G) \cdot \exp(\H^2(G, \Q/\Z))$. Another proof which goes via the decomposition into algebraic and transcendental Brauer group elements can be found in \cite[Cor.~6.3]{LS24}
\end{proof}

We can evaluate Brauer group elements of $BG(k)$.

\begin{definition} \label{def:evaluate}
	Let $\beta \in \Bre BG$ be represented by $1 \to \mu_n \to G_{\beta} \to G \to 1$. Then applying Galois cohomology to \eqref{seq:central} we obtain a connecting map
	$$\delta: \H^1(k,G) \to \H^2(k,\mu_n) = \Br k[n];$$	
	this exists exactly because $\mu_n$ is central, see \cite[\S5.7]{Ser02}. We define the evaluation $\beta(\varphi)$
	of $\beta$ at an element $\varphi \in BG[k]$ to be
	$$\beta: BG[k] \to \Br k, \quad \varphi \mapsto \delta(\varphi).$$
\end{definition}

\begin{lemma}\label{lem:evalutation_map_trivial_embedding_problem}
	We have $\beta(\varphi) = 0$ if and only if $\varphi$ lies in the image of the map $\H^1(k,G_\beta) \to \H^1(k,G)$.
\end{lemma}
\begin{proof}
	Immediate from the long exact sequence in cohomology of pointed sets.
\end{proof}

\begin{example} \label{ex:Q8}
\hfill
	\begin{enumerate} 
	\item Consider the central extension $\beta: 1 \to C_2 \to C_4 \to C_2 \to 1$. For $a \in k^{\times}/ k^{\times 2} \cong  \HH^1(k, C_2)$ we have $\beta(a) = (-1, a)$, where the brackets on the right-hand side denote the corresponding quaternion algebra. See \cite[Thm.~1.2.4]{Ser08} and in particular the subsequent proof.
	
	Note that if $-1 \in k^{\times 2}$ then $\beta(a) = 0$ for all $a \in \HH^1(k, C_2)$. This is an example of an \emph{algebraic} Brauer element (see \S \ref{sec:algebraic_Brauer_group}).
	 	
	\item For a more involved example we consider the Brauer element $\beta: 1 \to \mu_2 \to  Q_8 \to V_4 \to 1$. Consider an element $(a,b) \in  (k^{\times}/k^{\times 2})^2 \cong \HH^1(k, C_2)^2 \cong  \HH^1(k, V_4)$. Then $\beta((a,b)) = (-1, -1) + (-a, -b) \in \Br k$; this follows from the analysis in \cite[\S VI]{Witt1936}. 
	
    In this case $\beta$ is not algebraic; it is \emph{transcendental} (see \S \ref{sec:transcendental_Brauer_group}).
	\end{enumerate}
\end{example}

Recall from Definition \ref{def:integral_point} that $BG(\O_v)$ denotes the groupoid of unramified cocycles. The following is standard for Brauer groups of varieties \cite[Prop.~13.3.1]{CT21}.

\begin{lemma} \label{lem:trivial_on_O_v}
	Let $\beta \in \Bre BG$. Then $\beta$ evaluates trivially on $BG(\O_v)$ for all but finitely many $v$.
\end{lemma}
\begin{proof}
	Let $v$ be a place such that the $\Gamma_{k_v}$-group $G_\beta$ is unramified, in the sense of Definition \ref{def:unramified}.
	The map $BG_{\beta}(\mathcal{O}_v) \to BG(\mathcal{O}_v)$ is then surjective by Lemma \ref{lem:unramified_cocycle}, since we can lift the image of a topological generator from $G$ to $G_\beta$. This implies the result by Lemma \ref{lem:evalutation_map_trivial_embedding_problem}.
\end{proof}

\subsection{Markings}

We will make extensive use of the following extra structure on central extensions. 

\begin{definition} \label{def:marked_central_extension}
	Let $C \subset G(-1) \setminus \{e\}$ be a conjugacy and Galois invariant subset and $1 \to \mu_n \to G_{\beta} \to G \to 1$ a central extension. A $C$-\emph{marking on $G_{\beta}$} is a conjugacy and Galois invariant subset $A \subset G_{\beta}(-1)$ such that $A \to C$ is a (Galois equivariant) bijection.
	
	Let $\mathcal{C} \subset \mathcal{C}_G$ be the set of conjugacy classes contained in $C$. Note that $\mathcal{C}$ completely determines $C$, so we will also call these a $\mathcal{C}$-\emph{marking}.
	We denote by $\Br^{\mathrm{e}}_{\mathcal{C}} BG$ the collection of central extensions which admit a $\mathcal{C}$-marking. We call this the partially unramified Brauer group of $BG$ with respect to $\mathcal{C}$.
	
	A \emph{geometric marking} is a marking over $k^{\sep}$; this is a subset as above but without the requirement to be Galois invariant.
\end{definition}

This definition looks fairly mysterious. These Brauer groups $\Br^{\mathrm{e}}_{\mathcal{C}} BG$ play the role of the Brauer group of open subsets $U \subset X$ of a Fano variety as in \eqref{eqn:adjoint_rigid}. However $BG$ has no non-trivial open subsets so one cannot mimic this definition directly. The subset $\mathcal{C}$ plays the role of some kind of boundary divisor on a hypothetical compactification of $BG$. 

\begin{definition}
	In the case $C = G(-1) \setminus \{e\}$ we call the marking \emph{complete} and denote $\Brun BG := \Br^{\mathrm{e}}_{\mathcal{C}_G^*} BG$ which we call the \emph{unramified Brauer group}.
\end{definition}

We emphasise that here we only ask for \textit{existence} of a marking; we do not view the choice of marking as part of the data. In particular, the elements of $G_\beta(-1)$ admit a natural translation action of $\Z/n\Z = \mu_n(-1)$ given by 
\begin{equation} \label{def:torsor_action}
	G_\beta(-1) \times \Z/n\Z \to G_\beta(-1), \quad (\gamma,q) \mapsto \gamma + q:=(\,\zeta \mapsto	 \gamma(\zeta) \cdot \iota(\zeta^q)\,)
\end{equation}
where $\iota: \mu_n \to G_\beta$ is the central embedding of $\mu_n$ into $G_\beta$.
We study how this acts on markings; Recall that a \textit{torsor} under a group action is a non-empty set upon which the group acts freely and transitively.

\begin{lemma} \label{lem:Q/Z-torsor}
	Let $c \in \mathcal{C}_G$ and $\beta$ as in \eqref{def:central_extension_mu_n}.
	Then the collection of geometric 
	$c$-markings is a torsor
	under $\Z/n\Z$, provided it is non-empty. 
\end{lemma}
\begin{proof}
	As $\mu_n \subset G_\beta$ is central and lies in the kernel, it is clear that a geometric marking is sent 
	to a geometric marking. The action is also clearly free. For transitivity, let $A_1$ and 
	$A_2$ be geometric markings. After translating, we may assume that $A_1 \cap A_2 \neq \emptyset$.
	In which case $A_1 = A_2$. Indeed, as $c$ is a conjugacy class so 
	are $A_1$ and $A_2$, and any intersecting conjugacy classes are equal.
\end{proof}

The following fundamental theorem \cite[Thm.~7.4]{LS24} gives an arithmetic classification, via the partial adelic spaces from Definition \ref{def:adelic_space}, which makes clear their arithmetic significance even if Definition \ref{def:marked_central_extension} itself is mysterious. It is a strengthening of Lemma \ref{lem:trivial_on_O_v} provided a suitable marking exists.

\begin{theorem} \label{thm:trivial_on_O_v_R}
	Let $\beta \in \Bre BG$ and $\mathcal{C} \subset \mathcal{C}_G$ be Galois invariant. Then $\beta \in \Br^{\mathrm{e}}_{\mathcal{C}} BG$ if and only if $\beta$ evaluates trivially on $BG(\O_v)_{\mathcal{C}}$ for all but finitely many places $v$.
\end{theorem}
\begin{proof}[Proof sketch]
	The Brauer element $\beta$ is represented by a finite central extension $1 \to \mu_n \to G_{\beta} \to G \to 1$. Let $v$ be a tame non-archimedean place coprime to $n$. We will work with $[G(-1)/G](\F_v)$ instead of $BG(k_v)$, which is fine by Theorem \ref{thm:hensel}. Let $C \subset G(-1)$ be the union of conjugacy classes in $\mathcal{C}$.
		
	Assume that $\beta$ admits a marking. Let $(\gamma, F) \in [G(-1)/G](\F_v)$ with $\gamma \in C$ (which translates to the corresponding cocycle lying in $BG(\O_v)_{\mathcal{C}}$). The element $\beta$ evaluates trivially on $(\gamma, F)$ if and only if $(\gamma, F)$ lifts to a pair $(\gamma_{\beta}, F_{\beta}) \in [G_{\beta}(-1)/G_{\beta}](\F_v)$, by Lemma \ref{lem:evalutation_map_trivial_embedding_problem}. However, we may take $\gamma_{\beta}$ to be the marked element lying above $\gamma$ and $F_{\beta}$ any lift of $F$. The elements $\gamma_{\beta}$ and $F_{\beta} \Frob_v(\gamma_{\beta}) F_{\beta}^{-1}$ are both contained in the marking and map to the same elements since $(\gamma, F) \in [G(-1)/G](\F_v)$. They are thus equal, i.e.~$(\gamma_{\beta}, F_{\beta})$ is an element of $[G_{\beta}(-1)/G_{\beta}](\F_v)$. It follows that $\beta$ evaluates trivially on $(\gamma,F)$, as required.
	
	For the reverse implication, one needs to rule out the existence of a marking. This requires keeping careful track of the Galois actions and a suitable application of the Chebotarev density theorem. For a proof via stacks, see \cite[Thm.~7.4]{LS24}.
\end{proof}


\subsection{Residues} \label{sec:residues}

It is convenient to have a way to measure the failure of the existence of a marking. This is provided by the \emph{residue} of a Brauer group element. The residue is defined in more generality in \cite[Def.~5.25]{LS24}. We give a down-to-earth description in terms of central extensions.

Let $\beta$ be a Brauer group element represented by a central extension as in \eqref{def:central_extension_mu_n}. Let $\mathcal{C} \subset \mathcal{C}_G^*$ be a Galois orbit and $C \subset G(-1)$ be the corresponding conjugacy invariant subset. Then the first check for the existence of a $\mathcal{C}$-marking is whether there exists a conjugacy invariant subset $A \subset G_{\beta}(-1)$ which induces a bijection $A \to C$; i.e.~forgetting the Galois action. Recall from Definition \ref{def:marked_central_extension} that this is called  a \emph{geometric marking}.

We can associate to a geometric marking a $\Z/n\Z$-torsor via Lemma \ref{lem:Q/Z-torsor}. More precisely, let $c \in \mathcal{C}$ and recall that $k(c)$ denotes the residue field of $c$, as in Definition \ref{def:residue_field}, so that the Galois group $\Gamma_{k(c)}$ fixes $c$. For $A$ to be a marking, we require $A$ to be also invariant under the action of $\Gamma_{k(c)}$. We test this via an element of $\H^1(k(c),\Q/\Z)$ as follows: For each $\sigma \in \Gamma_{k(c)}$ we obtain a possibly different geometric marking $\sigma(A)$. However by Lemma~\ref{lem:Q/Z-torsor} this differs from the original marking by a unique element of $\Z/n\Z$, i.e.~there exists a  unique $q_\sigma \in \Z/n\Z$ such that $\sigma(A) = A + q_\sigma$. Moreover the centrality of $\mu_n$ implies that $q_{\sigma \tau} = q_{\sigma} + q_{\tau}$ for all $\sigma,\tau \in \Gamma_{k(c)}$. We thus obtain a well-defined homomorphism $\Gamma_{k(c)} \to \Z/n\Z$, i.e.~an element of $\H^1(k(c),\Z/n\Z) \subset \H^1(k(c),\Q/\Z)$, which by Lemma~\ref{lem:Q/Z-torsor} is independent of the choice of $A$. We formalise this process using the following definition.






\begin{definition} \label{def:residue}
	Let $\mathcal{C} \subset \mathcal{C}_G$ be a Galois orbit. We define the residue $\partial_{\mathcal{C}}(\beta)$ of $\beta$ along $\mathcal{C}$ as follows. If there is no geometric marking, then by convention we simply say that the residue $\partial_{\mathcal{C}}(\beta)$ is \textit{transcendental}.
	
	If a geometric marking exists, then we say that the residue $\partial_{\mathcal{C}}(\beta)$ is \textit{algebraic}
	and we define $\partial_{\mathcal{C}}(\beta) \in \H^1(k(c),\Q/\Z)$ to be the element constructed above for some choice of $c \in \mathcal{C}$.
\end{definition}

Note that the isomorphism class of the field $k(c)$ is independent of the choice of $c \in \mathcal{C}$. We do not require a precise definition of the residue in the transcendental case in this survey, but rather whether the residue is transcendental or not.

\begin{lemma} \label{lem:residue}
	A $\mathcal{C}$-marking exists if and only if the residue is algebraic and trivial.
\end{lemma}
\begin{proof}
	If a marking exists then by definition the residue is algebraic. It is also trivial,
	since the restriction of the marking to $c$ 
	is invariant under the action of $\Gamma_{k(c)}$.
	
	For the reverse implication, let $c \in \mathcal{C}$ and $A \subset G_{\beta}(-1)$
	be a geometric marking. As the residue is trivial $A$ is invariant under the action of $\Gamma_{k(c)}$. Then the $\Gamma_k$-orbit of $A$ is now a $\mathcal{C}$-marking in the sense of Definition \ref{def:marked_central_extension}.
\end{proof}

\subsection{The Brauer--Manin obstruction}
Recall the famous counter-example of Wang \cite{Wan48} to Grunwald's Theorem: There is no $\Z/8\Z$-extension of $\Q$ which realises the unramified $\Z/8\Z$-extension of $\Q_2$. A key difficulty in formulating a conjecture for the leading constant in Malle's conjecture is that one would like the leading constant to have an adelic interpretation, so should take into account such obstructions. In Example \ref{ex:Grunwald-Wang} we show that Wang's counter-example can be interpreted as a \textit{Brauer--Manin obstruction} on the stack $B(\Z/8\Z)$.

The Brauer--Manin obstruction was first introduced by Manin \cite{Man71} to use reciprocity laws to give obstructions to the Hasse principle on a variety $X$ over a number field $k$; this asks whether $X(\Adele_k) \neq \emptyset$ implies that $X(k) \neq \emptyset$. At the time it explained all known counter-examples to the Hasse principle. It was subsequently generalised to give obstructions to weak and strong approximation; this concerns the density properties of $X(k)$ in $\prod_v X(k_v)$ and $X(\Adele_k)$ respectively (see \cite[\S 13.3]{CT21} for further background). Since $BG(k) \neq \emptyset$ it is the latter properties we shall be interested in, but see \cite[\S 7.6]{LS24} and \cite{LT26} for examples of failures of the Hasse principle relevant to Malle's conjecture.

The Brauer--Manin set appears in  Peyre's constant \eqref{eqn:Peyre}, and it is this approach we take for the leading constant in Malle's conjecture. Fundamental to the Brauer--Manin obstruction is the short exact sequence

\begin{equation} \label{seq:fundamental_CFT}
	0 \to \Br k \to \bigoplus_v \Br k_v \to \Q/\Z \to 0
\end{equation}
from class field theory \cite[\S13.1.2]{CT21}. This exact sequence encodes all possible abelian reciprocity laws. For example on the level of $2$-torsion, it recovers Hilbert's reciprocity law for the Hilbert symbol, which for $k = \Q$ is equivalent to quadratic reciprocity. In this exact sequence, the final map is the sum over all local invariant maps $\inv_v \Br k_v: \Br k_v \to \Q/\Z$. These maps are isomorphisms for $v$ non-archimedean. In the archimedean case $\inv_v$ is an embedding and $\Br \C = 0$ and $\Br \R = \Z/2\Z$.

\begin{definition} \label{def:BM_pairing}
	We define the Brauer--Manin pairing via
	$$\langle \cdot, \cdot \rangle_{\mathrm{BM}} : \Bre BG \times BG(\Adele_k) \to \Q/\Z, \quad (\beta, (\varphi_v)) \mapsto \sum_v \inv_v \beta(\varphi_v).$$
\end{definition}

To be completely explicit recall from Definition \ref{def:evaluate} how to evaluate a Brauer group element $\beta$ via connecting map in Galois cohomology. Then $\langle \beta, (\varphi_v) \rangle_{\mathrm{BM}}$ is equal to the image of $(\varphi_v)$ with respect to the composition
\[
BG(\Adele_k) \to \bigoplus_v \H^2(k_v,\mu_n) = \bigoplus_v \Br(k_v) \to \Q/\Z
\] 
where the last map is the sum over all local invariants. It is clear from \eqref{seq:fundamental_CFT} that any element of $BG[k]$ lies in the left kernel of the Brauer--Manin pairing.

The sum which appears in Definition \ref{def:BM_pairing} is finite by Lemma \ref{lem:trivial_on_O_v}. Note that there is not a well-defined pairing between $\Bre BG$ and $\prod_v BG(k_v)$ in general. Still, given Theorem \ref{thm:trivial_on_O_v_R}, we can obtain a well-defined pairing providing we impose markings on the Brauer group elements.

\begin{lemma} \label{lem:Brauer-Manin_well_defined}
	Let $\mathcal{C} \subset \mathcal{C}_G$ be Galois invariant. Then the pairing
	$$\Br^{\mathrm{e}}_{\mathcal{C}} BG \times BG(\Adele_k)_{\mathcal{C}} \to \Q/\Z, \quad (\beta, (\varphi_v)) \mapsto \sum_v \inv_v \beta(\varphi_v)$$
	is well-defined and continuous on the right. In particular, the pairing 
	$$\Brun BG \times \prod_v BG(k_v), \quad (\beta, (\varphi_v)) \mapsto \sum_v \inv_v \beta(\varphi_v)$$
	is well-defined.
\end{lemma}
\begin{proof}
	Immediate from Theorem \ref{thm:trivial_on_O_v_R}.
\end{proof}

\begin{definition}
	For subsets $W \subset BG(\Adele_k)_{\mathcal{C}}$ and $\mathscr{B} \subset \Br^{\mathrm{e}}_{\mathcal{C}} BG$ we denote by
	$$ W^{\mathscr{B}} := \{ \varphi \in W: \langle b, \varphi \rangle_{\mathrm{BM}} = 0 \text{ for all }b \in \mathscr{B} \}$$
	the associated \textit{Brauer--Manin set}. 
	We say that $\mathscr{B}$ \emph{induces a Brauer--Manin obstruction on $W$} if $W \neq W^{\mathscr{B}}$.
	If clear from context, we denote by 
	$BG(\Adele_k)_{\mathcal{C}}^{\Br} := BG(\Adele_k)_{\mathcal{C}}^{\Br^{\mathrm{e}}_{\mathcal{C}} BG}$.
\end{definition}

For an example of the Brauer--Manin obstruction, see Example \ref{ex:Grunwald-Wang}.

\subsection{Algebraic Brauer group elements} \label{sec:algebraic_Brauer_group}
We define the \emph{algebraic Brauer group} of $BG$ to be
\begin{equation}
	\Br^{\mathrm{e}}_1 BG := \ker(\Bre BG \to \Bre BG_{k^{\sep}})
\end{equation}
where $BG_{k^{\sep}}$ denotes $BG$ viewed over the separable closure of $k$. The advantage of the algebraic Brauer group is that it is easier to calculate and work with. Also in practice algebraic Brauer group elements are sufficient for most applications. There is a completely explicit construction of algebraic Brauer group elements.

\begin{definition} \label{def:algebraic_central_extension}
	Let $n$ be coprime to $p$ and $\alpha \in \H^1(k, \Hom(G, \mu_n))$. We define the central extension
	$$1 \to \mu_n \to E_{\alpha} \to G \to 1$$
	by considering the abstract group $E_{\alpha} = \mu_n \times G$ but equipping it with the twisted Galois action $\sigma \cdot_{\alpha} (\zeta, g) = (\sigma(\zeta) \cdot \alpha(\sigma)(\sigma(g)), \sigma(g))$ for $\sigma \in \Gamma_k$ and $(\zeta, g) \in E_{\alpha} := \mu_n \times G$.
\end{definition}
It is straightforward but tedious to check that the map $\alpha \to E_{\alpha}$ is additive with respect to the Baer sum.

\begin{lemma}\label{lem:algebraic_Brauer_H1_Pic}
	Let $n \in \N$ with $\exp(G) \mid n$. The map 
	$$\HH^1(k,\dual{G}) = \HH^1(k, \Hom(G, \mu_{n})) \to \Br^{\mathrm{e}}_1 BG$$
	is an isomorphism.
\end{lemma}
\begin{proof}
	For injectivity we use Lemma \ref{lem:trivial_in_Brauer_group}. Let $\chi: E_{\alpha} \to \mu_{\infty}$ be a homomorphism which is the identity on $\mu_n \subset E_{\alpha}$. Then as a morphism of abstract groups we have $\chi(\zeta, g) = \zeta \cdot \psi(g)$ for some homomorphism $\psi: G \to \mu_{\infty}$ of abstract groups. The condition on $n$ ensures that $\psi$ factors through $\mu_n$, i.e.~$\psi \in \Hom(G, \mu_n)$.
	
	The condition that $\chi$ is Galois equivariant implies that for all $\sigma \in \Gamma_k$ we have 
	\[
	\sigma(\zeta) \cdot \sigma(\psi)(\sigma(g))= \sigma(\chi(\zeta, g)) = \chi(\sigma \cdot_{\alpha} (\zeta, g)) = \sigma(\zeta) \cdot  \alpha(\sigma)(\sigma(g)) \cdot \psi(\sigma(g)). 
	\]
	
	This implies that $\alpha = \sigma(\psi) \psi^{-1}$, in other words, $\alpha$ is equal to the coboundary of $\psi$, so represents the trivial cohomology class.
	
	Let now $\beta \in \Br^{\mathrm{e}}_1 BG$. The assumption that $\beta$ is algebraic implies that the central extension splits as an extension of abstract groups. We thus have $G_{\beta} = \mu_{\infty} \times G$ as an abstract group. The subgroup $\mu_{\infty} = \mu_{\infty} \times \{1\}$ is preserved by the Galois action and the Galois action is equivariant with respect to the projection to $G$. This implies that for all $g,\sigma$ there exists $a_{\sigma, g} \in \mu_{\infty}$ such that $\sigma(\zeta, \sigma^{-1}(g)) = (\sigma(\zeta) a_{\sigma, g}, g)$ for all $\zeta \in \mu_{\infty}$. One can now check the following via straightforward cocycle computations.
	\begin{enumerate}
		\item For all $\sigma \in \Gamma_k$ and $g, h \in G$ we have $a_{\sigma, gh} = a_{\sigma, g} a_{\sigma, h}$. This implies that the map $g \to a_{\sigma, g}$ is an element of $\Hom(G, \mu_{\infty})$ for all $\sigma \in \Gamma_k$. It in particular factors through $\Hom(G, \mu_n) = \Hom(G, \mu_{\infty})$.
		\item For all $\sigma, \tau \in \Gamma_k$ and $g \in G$ we have $a_{\sigma \tau, g} = a_{\sigma, g} a_{\tau, \sigma^{-1}(g)}$. 
		\end{enumerate}
		These claims imply that 
		\[
		\alpha: \Gamma_k \to \Hom(G, \mu_n): \sigma \to (g \to a_{\sigma, g}).
		\] 
		is a cocycle. It then follows immediately from the definition that $G_{\beta} = E_{\alpha}$.
\end{proof}
One can compute the evaluation map coming from $E_{\alpha}$ directly in terms of cup products. Consider the tautological pairing $\Hom(G, \mu_n) \times G^{\mathrm{ab}} \to \mu_n$. We get an induced cup product pairing
\[
\HH^1(k, \Hom(G, \mu_n)) \times \HH^1(k, G^{\mathrm{ab}}) \to \HH^2(k, \mu_n) = \Br k[n] : (\alpha, \chi) \to \alpha \cup \chi.
\]
If we take $\exp(G) \mid n$ then we have $\HH^1(k, \Hom(G, \mu_n)) \cong \HH^1(k , \hat{G})$ by Lemma \ref{lem:algebraic_Brauer_H1_Pic}. We then have the following.

\begin{lemma} \label{lem:cup_products}
	The diagram 
\[
\xymatrix{ \Br^{\mathrm{e}}_1 BG \times BG(k)  \ar[r] \ar[d] & \Br k \ar[d] \\ 
\H^1(k, \dual{G}) \times \H^1(k,G^{\mathrm{ab}}) \ar[r] & \H^2(k, \mu_n ) }
\]
	commutes, where the top row is evaluation of a Brauer group element and the
	bottom row is induced by the cup product.
\end{lemma}
\begin{proof}
	The proof is rather technical, we refer to \cite[Lem.~6.4]{LS24} for the details.
\end{proof}

We have introduced Brauer group elements via central extensions as it allows a uniform treatment of algebraic and transcendental Brauer group elements. However for algebraic Brauer group elements it is often easier to think about them via cup products, as in Lemma \ref{lem:cup_products}, since then the local invariants are easier to calculate. This is illustrated in the following example.

\begin{example}[Grunwald--Wang] \label{ex:Grunwald-Wang}
We will now explain the famous counterexample of Grunwald-Wang that there exists no $C_8$-extension of $\Q$ which is totally inert and unramified at $2$, via the Brauer-Manin obstruction.

Let $G = C_8$ and $\beta = 16 \in \Q^{\times}/\Q^{\times 8} \cong \HH^1(\Q, \mu_{8}) = \Br^{\mathrm{e}}_1 BC_8$, where the first isomorphism is by Kummer theory. We will show that this Brauer element induces an obstruction to weak approximation.

Let $(\varphi_v)_v \in \prod_v BC_8[\Q_v]$ be such that $\varphi_2: \Gamma_2 \to C_8$ is a surjective unramified character. We will evaluate the local invariants $\text{inv}_v(\beta(\varphi_v))$. Note first that $\beta(\varphi_v) = 16 \cup \varphi_v$ by Lemma \ref{lem:cup_products}. 

We have $16 = \sqrt{2}^8= \sqrt{-2}^8= (1 + \sqrt{-1})^8$ so $16 = 1 \in \HH^1(\Q_v, \mu_8)$ for all $v \neq 2$ since for all odd or real places at least one of $\sqrt{2}, \sqrt{-2}, \sqrt{-1}$ is contained in $\Q_v$. For $v = 2$ we compute by using the relation between cup products and Hilbert symbols (see \cite[\S XIV.1]{Ser79}) that
\[
	16 \cup \varphi_v = 2^4 \cup \varphi_v = (2, 5)_2.
\]
The quaternion algebra $(2, 5)_2$ is non-trivial so $\text{inv}_2(\beta(\varphi_2)) = \frac{1}{2}$. To summarize, we have computed that 
\[
\sum_{v} \text{inv}_v(\beta(\varphi_v)) = \text{inv}_2(\beta(\varphi_2)) = \frac{1}{2} \neq 0
\]
so there is a Brauer-Manin obstruction to a global cocycle $\varphi$ realizing such a collection of local cocycles $(\varphi_v)_v$. Translating this into the language of field extensions we exactly get the Grunwald-Wang counterexample.
\end{example}

\begin{remark}[The residue]
	Via Definition \ref{def:residue}, for each $c \in \mathcal{C}_G^*$ we obtain a residue
	map
	$$\partial_{c}: \Br^{\mathrm{e}}_1 BG \to \H^1(k(c),\Q/\Z)$$
	which is a homomorphism. Various alternative definitions of the residue, as well as ways of calculating it, can be found in \cite[\S6.6]{LS24}.
\end{remark}

\subsection{Transcendental elements} \label{sec:transcendental_Brauer_group}

\begin{definition}
An element of $\Bre BG$ not in $\Br^{\mathrm{e}}_1 BG$, i.e.~which is not algebraic, is called \emph{transcendental}. 
We call $\Bre BG_{k^{\sep}}$ the \emph{geometric Brauer group}.
\end{definition}

We use analogous notation to Definition \ref{def:marked_central_extension} for markings.
Transcendental Brauer group elements are fairly mysterious in general. By definition there is an exact sequence
\begin{equation} \label{seq:Br_extension}
	0 \to \Br^{\mathrm{e}}_{\mathcal{C},1} BG \to \Br^{\mathrm{e}}_{\mathcal{C}} BG \to \Br^{\mathrm{e}}_{\mathcal{C}} BG_{k^{\sep}}.
\end{equation}
However there is no reason for the right hand map to be surjective; even if the codomain is restricted to Galois invariant elements. Moreover the sequence need not split.

But at least the geometric Brauer group is a well-studied object, namely by Definition \ref{def:central} elements of $\Bre BG_{k^{\sep}}$ are central extensions of abstract groups
$$ 0 \to \Q/\Z \to E \to G \to 0.$$
Such central extensions are classified by a well-studied object in group theory, namely the \textit{Schur multiplier} $\H^2(G, \Q/\Z) = \H^2(G, \C^\times)$ of $G$, hence $\Bre BG_{k^{\sep}} \cong \H^2(G, \Q/\Z)$  (see \cite[Thm.~IV.3.12]{Bro94}).

Therefore calculating the Schur multiplier is a first step to understanding the transcendental Brauer group. However in Malle's conjecture one is typically interested in marked central extensions, and the marked geometric Brauer group can be trivial even if the Schur multiplier is not.

\begin{example} \hfill

	(1) Let $G$ be abelian. Then the
	geometric Brauer group is $\H^2(G,\Q/\Z)$, which is typically non-trivial.
	However, in Conjecture \ref{conj:balanced} we are interested in marked
	central extensions where the marking generates $G$. It turns out that any
	such marked central extension is trivial \cite[Lem.~6.38]{LS24}. Thus the partially
	unramified transcendental Brauer group in Conjecture \ref{conj:balanced} is trivial for $G$ abelian.

	(2) The following is the smallest example of a non-trivial transcendental Brauer group element
	with a balanced marking. Take $G=A_4$. We have $\PSL_2(\F_3) \cong A_4$, thus we have a 
	central extension
	$$0 \to \mu_2 \to \SL_2(\F_3) \to A_4 \to 0$$
	which hence determines an element of order dividing $2$ in $\Bre BA_4$.
	We take $C \subset A_4$ to be the elements of order $3$; note that these generate $A_4$.
	Then a $C$-marking is obtained by taking the elements of order $3$ in $\SL_2(\F_3)$, of which 
	there are $8$ and they form a single conjugacy class (this subset is closed under the Galois action as it is closed under powering). This yields a non-trivial
	transcendental Brauer group element of $BA_4$ of order $2$ with a balanced marking; 
	see \cite[\S 10.6]{LS24}
	for more details as well as arithmetic implications.
	
	This example does not admit a complete marking: assume for a contradiction that there is a marking
	over the conjugacy class of elements of order $2$ in $A_4$. But there is a unique non-central conjugacy class of
	$\SL_2(\F_3)$ consisting of elements of $2$-power order, namely the elements of order $4$. However
	this conjugacy class has size $6$, so it cannot biject with the elements of order $2$ in $A_4$,
	of which there are $3$. 
	
	(3) The above examples hint that it may be difficult to come up with examples of finite groups
	$G$ with a non-trivial unramified transcendental Brauer group. This is indeed the case;
	the first such example is due to Saltman \cite{Sal82} for a group of order $p^9$.
	Such examples were put into a more
	general framework by Bogomolov \cite{Bog87}, in terms of the so-called
	\emph{Bogomolov multiplier} (see \cite[Rem.~6.35]{LS24} for more details).
\end{example}




\subsection{Computing the Brauer group}
Let $C \subset G(-1)$ be Galois and conjugacy invariant and generate $G$, with corresponding $\mathcal{C} \subset \mathcal{C}_G$. We now explain a process to calculate the Brauer group $\Br^{\mathrm{e}}_{\mathcal{C}} BG$ as outlined in \cite[\S6.7]{LS24}; a closely related procedure is also described in \cite{Lag25} in the case of a complete marking. Calculating this Brauer group is necessary to calculate the leading constant.

To begin we recall from Lemma \ref{lem:exp} that any element of $\Br^{\mathrm{e}}_{\mathcal{C}} BG$ may be represented by a central extension, with $m := \exp(G) \cdot \exp(\HH^2(G, \Q/\Z))$ 
\begin{equation} \label{eqn:central_extension_sequence}
1 \to \mu_{m} \to E \to G \to 1.
\end{equation}
First enumerate all such central extensions as finite groups, i.e.~without taking into account the $\Gamma_k$-action. There are finitely many such extensions as the group cohomology $\H^2(G,\mu_m)$ is finite for any $m$. They can be explicitly found using, for example, databases of small groups.

Next we need to take into account the existence of a marking, in the sense of Definition \ref{def:marked_central_extension}. To do this we view $C$ as a subset of $G$ by using an isomorphism as in Lemma \ref{lem:Galois_action_on_G(-1)}; the choice of isomorphism does not affect the existence of a marking. One then enumerates all conjugacy invariant subsets $A \subset E$ such that $A \to C$ is a bijection; this yields a geometric marking in the sense of Definition \ref{def:marked_central_extension}.

We next need to determine whether there is a Galois action on $E$ which both
\begin{enumerate}
	\item Makes \eqref{eqn:central_extension_sequence} an exact sequence of $\Gamma_k$-groups
	\item Leaves $A$ Galois invariant with respect to the anticyclotomic Galois action from Definition \ref{def:G(cycl)}; for example if $k = \Q$, it needs to be invariant under invertible powering.
\end{enumerate}
Note that if a Galois action on $E$ satisfying (2) exists, then it is necessarily unique; indeed $C$ generates $G$ by assumption, so it follows that $A$ and $\mu_{m}$ together generate $E$. So the Galois action on $A$ uniquely determines the Galois action on $E$. 
To find such a Galois action, one considers all possible subgroups of $\Aut E$ and determines whether it satisfies the required properties (1) and (2) above, then performs a Galois twist by this Galois action. 

This computes marked central extensions. We can keep track of the change of choice of marking using Lemma \ref{lem:Q/Z-torsor}. Finally, since we have central extensions with finite kernel, rather than $\mu_\infty$, it is possible that there is redundancy amongst the elements obtained. One can find a basis for the Brauer group using the criterion from Lemma \ref{lem:trivial_in_Brauer_group}.

We note that the uniqueness of the Galois action after fixing the marking implies the following finiteness result, since the collection of all central extensions \eqref{eqn:central_extension_sequence} is finite.

\begin{corollary} \label{cor:Br_finite}
	Let $\mathcal{C} \subset \mathcal{C}_G$ be Galois invariant and generate $G$. Then $\Br^{\mathrm{e}}_{\mathcal{C}} BG$ is finite.
\end{corollary}

This result is also proven in \cite[Cor.~6.30]{LS24} using cohomological techniques. The method we describe does not distinguish between algebraic and transcendental Brauer group elements. Taking these into account and utilising the sequence \eqref{seq:Br_extension} can naturally help to speed up the process, for example by calculating residues and using Lemma \ref{lem:residue}.

\section{The leading constant} \label{sec:leading_constant}

We now have all the ingredients we need to state our conjectures on the leading constant in Malle's conjecture. We state our conjectures in terms of elements of $BG[k]$, equivalently in terms of cocycles with values in $G$. One can also state a version in terms of possibly non-Galois number fields but care is required keeping track of groupoid cardinalities; see \cite[Conj.~9.3]{LS24}.

\subsection{Balanced heights}

Our main conjecture concerns balanced heights as in Definition \ref{def:balanced}. We assume that $k$ is a number field.

\begin{conjecture} \label{conj:balanced}
	Let $G$ be a finite $\Gamma_k$-group and $H$ a big balanced height function on $BG$. Then there exists a thin subset $\Omega \subset BG[k]$ such that
	$$\frac{1}{|Z(G)^{\Gamma_k}|}\#\{ \varphi \in BG[k] \setminus \Omega : H(\varphi) \leq B\} \sim c(k,G,H) B^{a(H)} (\log B)^{b(k,H)-1}$$
	where 
	\begin{align*}
	a(H) & = (\min_{c \in \mathcal{C}_G^*}w(c))^{-1},   \quad
	\mathcal{M}(H) = \{ c \in \mathcal{C}_G^* : w(c) = a(H)^{-1}\}, \\
	 b(k,H) & = \#\mathcal{M}(H)/\Gamma_k,
	\end{align*}
	and 
	$$c(k,G,H) = \frac{a(H)^{b(k,H)-1}\cdot |\Br^{\mathrm{e}}_{\mathcal{M}(H)} BG| \cdot
	\tau_H( BG(\Adele_k)_{\mathcal{M}(H)}^{\Br})}{\#\dual{G}^{\Gamma_k} (b(k,H) - 1)!}.$$
\end{conjecture}

Recalling Definition \ref{def:BG}, the left-hand side counts elements of $\H^1(k,G)$ of bounded height, away from a suitable thin set. If $G$ has trivial Galois action, then this is the same as continuous homomorphisms $\Gamma_k \to G$, up to conjugation. We sometimes use the notation $a_G(H), b_G(k,H)$ for the exponents to make the group $G$ clear. 
The prediction for the leading constant is modelled on \eqref{eqn:adjoint_rigid}. Briefly the analogue of the effective cone constant in this setting is the rational number $a(H)^{b(k,H)}/\#\dual{G}^{\Gamma_k}$; see \cite[Lem.~3.35]{LS24}. The Brauer group and Tamagawa factors are analogous; though since we do not have any notion of non-trivial open subset of $BG$, we instead use the Galois orbits of $\mathcal{C}_G$ to play the role of divisors in some hypothetical compactification. It is also shown in \cite[Lem.~3.26]{LS24} that the exponents which appear are as expected from Conjecture \ref{conj:Manin}.

Here $Z(G)$ denotes the \emph{centre} of $G$. This is included because we should count using the groupoid cardinality, and away from a thin set the automorphism group of an element $\varphi$ is exactly the centre (see \cite[Lem.~2.9]{LS24}).

As in Manin's conjecture, we suggest to remove a thin set to get the correct asymptotic formula. Firstly it should be clear that one should only consider surjective cocycles. However it is also necessary to remove more cocycles in general, as the following counter-example of Kl\"{u}ners \cite{Klu05} demonstrates. 
	
\begin{example} \label{ex:Kluners}
	Take $G = C_3 \wr C_2$ as in Example \ref{ex:wreath} with the discriminant height $\Delta$ as in
	Example \ref{ex:discriminant}. We have $a_G(\Delta) = b_G(\Q,\Delta) = 1$. 
	However the problematic thin set is given by the fields containing $\Q(\sqrt{-3})$, 
	which are parametrised by $B \mathrm{R}_{\Q(\sqrt{-3})/\Q} C_3$. We have
	$a_{\mathrm{R}_{\Q(\sqrt{-3})/\Q} C_3}(\Delta) = 1$ but $b_{\mathrm{R}_{\Q(\sqrt{-3})/\Q} C_3}(\Q,\Delta) = 2$,
	so these contribute a larger than expected power of $\log B$.
\end{example}

\begin{question}
	It is known that Manin's conjecture is compatible with respect to the Weil restriction \cite{Lou15}. Is the same true for Conjecture \ref{conj:balanced}? The Weil restriction naturally arises in Malle's conjecture when considering wreath products (see Example \ref{ex:wreath}).
\end{question}

\begin{remark}[The choice of thin set $\Omega$] \label{rem:thin_set}
	It is useful for the theory and  examples to allow flexibility in the choice of thin set $\Omega$. This may seem unsatisfactory, however actually once one has found a thin set which works, so does any larger thin set. This is the content of \cite[Thm.~9.13]{LS24}. Despite all this, we \emph{do} also have a conjecture about a choice of thin set which works. These should be the breaking cocycles as defined in \cite[\S 3.7]{LS24}.
	
	When $G$ has trivial Galois action, these are contained in an explicit set (see \cite[Thm.~3.40]{LS24}). Namely, we conjecture that Conjecture \ref{conj:balanced} holds with
\begin{equation} \label{eqn:thin_set}
\Omega =  \left\{ \varphi: \Gamma_k \to G :
	\begin{array}{ll}
	 	\varphi \text{ is not surjective or } \\
	   k_{\varphi} \text{ is not linearly disjoint to } k(\mu_{\exp(G)})
	   \end{array} \right\}.
\end{equation}
\end{remark}

\subsubsection{Volume of the Brauer--Manin set}
If $\mathcal{M}(H) \neq \mathcal{C}_G^*$, then by Theorem \ref{thm:trivial_on_O_v_R} the Brauer-Manin obstruction can  play a role at infinitely many places. Computing the volume of the Brauer-Manin set is then not straightforward from the definition. It is therefore useful to take a Fourier-theoretic perspective.

Firstly, note that the formula $\alpha \to e^{2 i \pi  \alpha}$ defines a well-defined map $\Q/\Z \to \C$. The following is thus well-defined. 

\begin{definition} \label{def:Brauer_transform}
	Let $H$ be a height, $\beta \in \Br^{\mathrm{e}}_{\mathcal{M}(H)} BG$ and $U \subset BG[\Adele_k]_{\mathcal{M}(H)}$ a Borel subset. The \emph{global Brauer transform} is defined as the integral
	\begin{equation*}
		\hat{\tau}_H(\beta ; U) = \int_U e^{2 i \pi \langle \beta, x \rangle_{\mathrm{BM}}} d\tau_H(x).
	\end{equation*}
\end{definition}
These global Brauer transforms are Euler products of local Brauer transforms.
\begin{lemma}
	Assume that $U = \prod_v U_v$ with $U_v \subset BG[k_v]$. We then have
	\[
	\hat{\tau}_H(\beta ; U) = \L^*(\mathcal{M}(H),1)\prod_v \lambda_v^{-1} \hat{\tau}_{H, v}(\beta, U_v).
	\]
	Where $\hat{\tau}_{H, v}(\beta, U_v)$ is the \emph{local Brauer transform}, defined as
	\[
	\int_{U_v} e^{2 i \pi \inv_v( \beta( \varphi_v))	} d\tau_{H, v}(x_v) = \sum_{\varphi_v \in U_v} \frac{e^{2 i \pi \inv_v( \beta( \varphi_v))}}{|\Aut(\varphi_v)| H_v(\varphi_v)^{a(H)}}.
	\]
\end{lemma}
\begin{proof}
	Immediate from the definition $\tau_H := \L^*(\mathcal{M}(H),1) \prod_v \lambda_v^{-1} \tau_{H, v}$ and the fact that the global Brauer-Manin pairing is the sum of the local Brauer-Manin pairings.
\end{proof}

\begin{lemma}\label{lem:finite_sum_Euler_products}
	Let $B \subset \Br^{\mathrm{e}}_{\mathcal{M}(H)} BG$ be a finite subgroup and $U \subset BG[\Adele_k]_{\mathcal{M}(H)}$ a Borel subset. We then have
	\[
	\tau_H(U^B) =  \frac{1}{|B|}\sum_{\beta \in B} \hat{\tau}_H(\beta ; U). 
	\]
\end{lemma}
\begin{proof}
	Immediate from character orthogonality.
\end{proof}

So we find that the Tamagawa factor in Conjecture \ref{conj:balanced} is a finite sum of Euler products. 

\begin{remark}[Fair vs balanced]
In Wood \cite[\S2.1]{Woo10} introduces the notion of fair height function for abelian groups. She does this via an algebraic condition, namely she calls the height function \emph{fair} on an abelian group $G$ if the minimal weight elements generate $G[r]$ for all $r$. She proves an asymptotic formula for such heights in \cite[Thm.~3.1]{Woo10}, with the leading constant given by a finite sum of Euler products over $k$. The finite sum of Euler products arises in her setting via the Grunwald--Wang theorem which, as explained in Example~\ref{ex:Grunwald-Wang}, can be interpreted as a Brauer--Manin obstruction, thus this is compatible with the expression in Lemma \ref{lem:finite_sum_Euler_products}. 

However there are other heights with nice leading constants. For example, Tavernier has given in  \cite[Ex.~1.7]{Tav25} an example of a balanced height on $\Z/6\Z$ whose leading constant is given by a single Euler product, but is not fair in the sense of Wood as it does not satisfy her algebraic condition. 


Balanced heights look like the correct class which encapsulate the properties one would like for the leading constant, but whilst also accommodating for possible Brauer--Manin obstructions and removing a thin set if necessary. One crucial difference however is that Wood obtains a finite sum of Euler products where the Euler factors agree at all but finitely many places, but for general balanced heights we obtain a finite sum of Euler products where the Euler factors can differ at infinitely many places. See \cite[Ex.~1.8(1)]{Tav25} for an explicit such example for $\Z/4\Z$. We explain how to interpret Wood's algebraic condition in terms of Brauer groups in \cite[\S10.7.4]{LS24}; briefly Wood's condition implies that the partially unramified Brauer group equals the unramified Brauer group, so by Theorem~\ref{thm:trivial_on_O_v_R} it can only play a role at finitely many places.
\end{remark}

\subsection{Heuristic for the conjectural asymptotic} \label{sec:heuristic}
We now give an illustrative heuristic explanation for Conjecture \ref{conj:balanced}. A simplified version of this heuristic for the discriminant, not taking into account the Galois action on $G(-1)$, appears in \cite[\S4]{Mal04}. We upgrade the ideas underlying the heuristic presented here to a precise conjecture in \S\ref{sec:Bhargava}. We assume that the weight takes integer values and that $k = \Q$.
	
Recall from \eqref{def:product_local_heights}
that the height is a product of local heights. The distribution of these integers 
is controlled by the minimal $p$-adic valuations of the local heights, which occur when the 
ramification type has minimal weight. Assuming no other structure and ignoring higher powers of $p$, we therefore expect that
$H(\varphi)$ behaves like a random integer of the form $n^{\min w} = n^{1/a(H)}$, where $n$ is squarefree.
However there may be many $\varphi$ with the same height; indeed this happens precisely when
the weight function takes minimal value. By Remark \ref{rem:surjective_ramification_type},
the possible conjugacy classes
of minimal weight which can be realised at a given good prime $p$ are exactly $\mathcal{M}(H)^{\Gamma_{\Q_p}}$. Assuming no other structure we expect
that each such conjugacy class occurs equally often.
Therefore, we expect that the quantity in Conjecture \ref{conj:balanced} behaves like
$$\sum_{n \leq B^{a(H)}} \mu^2(n)f(n), \quad \textrm{ where } f(n) = \prod_{p \mid n} |\mathcal{M}(H)^
{\Gamma_{\Q_p}}|$$
where $\mu$ denotes the M\"{o}bius function. We have the following asymptotic, using the Artin $\L$-functions from \S \ref{sec:global_Tamagawa}.

\begin{lemma} \label{lem:mu_f}
	\begin{align*}
	&\sum_{n \leq B^{a(H)}} \mu^2(n)f(n) \\
	&\sim \frac{a(H)^{b(k,H)-1} \L^*(\mathcal{M}(H),1)}{(b(k,H)-1)!} \prod_p \lambda_p^{-1}\left(1 + \frac{\#\mathcal{M}(H)^{\Gamma_{\Q_p}}}{p}\right) B^{a(H)}(\log B)^{b(k,H)-1}
	\end{align*}
	as $B \to \infty$
	\end{lemma}
\begin{proof}
We shall be brief as we give similar arguments in more depth in \S\ref{sec:Bhargava}.
The function $f$ is multiplicative, so we have the Euler product.
$$\sum_{n=1}^\infty \frac{\mu^2(n)f(n)}{n^s} = \prod_p
\left(1 + \frac{\#\mathcal{M}(H)^{\Gamma_{\Q_p}}}{p^s}\right).$$
Unravelling definitions, this is easily seen to equal $\L(\mathcal{M}(H),s) G(s)$,
where $G$ is holomorphic on $\re s \geq 1$. From Lemma \ref{lem:Dedekind_zeta}, the $\L$-function is given by a product of Dedekind zeta functions, with one zeta function per element of $\mathcal{M}(H)/\Gamma_k$. Thus it has a pole of order $b(k,H)$ at $s = 1$ and no other poles for $\re s \geq 1$. Applying a Tauberian theorem (e.g.~\cite[Thm.~A]{PTBZ25}) one obtains
$$\sum_{n \leq x} \mu^2(n)f(n) \sim \frac{\L^*(\mathcal{M}(H),1)}{(b(k,H)-1)!} \prod_p \lambda_p^{-1}\left(1 + \frac{\#\mathcal{M}(H)^{\Gamma_{\Q_p}}}{p}\right) \cdot x (\log x)^{b(k,H)-1}.$$
Taking $x = B^{a(H)}$ proves the result.
\end{proof}

We obtain from Lemma \ref{lem:mu_f} an expression very close to the formula in Conjecture~\ref{conj:balanced}, on using the mass formula \eqref{eqn:measure_minimal_weight}.

The key differences are the factor $1/\#\dual{G}^{\Gamma_k}$ and the Brauer--Manin obstruction. The factor $1/\#\dual{G}^{\Gamma_k}$ may look somewhat innocuous, however finding this rational number was an open problem which Bhargava explicitly posed in \cite[(8.6)]{Bha07} for its value. As already explained, it appears as part of Peyre's effective cone constant.

The heuristic can actually be modified to take into account possible Brauer--Manin obstructions: these would imply that some combinations of the $p$-adic values of the ramification type are not actually realised by some global $\varphi$. To account for these, we expect that $BG[k]$ is dense in $BG(\Adele_k)_{\mathcal{M}(H)}^{\Br}$, where the Brauer--Manin obstruction is taken with respect to $\Br^{\mathrm{e}}_{\mathcal{M}(H)} BG$ (see Conjecture \ref{conj:dense}). However $\Br^{\mathrm{e}}_{\mathcal{M}(H)} BG$ is finite by Corollary \ref{cor:Br_finite}, so one can cut out the Brauer--Manin set using character orthogonality, as in Lemma \ref{lem:finite_sum_Euler_products}. Moreover by Theorem \ref{thm:trivial_on_O_v_R} each such element takes the trivial value on the minimally ramified elements of $BG[k_v]$. Hence these Brauer group elements only play a role for the minimally ramified elements at possibly finitely many primes, which does not affect the overall asymptotic behaviour of the sum. We give full details when we upgrade this heuristic to a precise conjecture in \S \ref{sec:Bhargava}, which also formalises Bhargava's heuristic.
		


\subsection{Unbalanced heights} For unbalanced heights, we reduce to the balanced case via the \textit{Iitaka fibration}, motivated by the case of Manin's conjecture as explained in \S \ref{sec:leading_constant_Manin}. We take a definition of Iitaka fibration in our setting which seems most natural and gives the expected behaviour.

\begin{definition}
	Let $H$ be a height function with weight function $w$. Let $\mathcal{M}(w)$ be the minimal weight conjugacy classes. We define the associated \textit{Iitaka fibration} to be $G \to G/\langle \mathcal{M}(w) \rangle$ the quotient of $G$ via the subgroup generated by $\mathcal{M}(w)$.
\end{definition}

The restriction of $H$ to the fibres of the Iitaka fibration is now balanced. We sort the counting problem according to the fibres of the Iitaka fibration, with the expectation that the sum over all the leading constants converges. This leads to the following prediction.

\begin{conjecture} \label{conj:non_balanced}
	If $M := \langle \mathcal{M}(w) \rangle \neq G$ then there exists a thin subset $\Omega \subset BG[k]$ such that
	$$\frac{1}{|Z(G)^{\Gamma_k}|}\#\{ \varphi \in BG[k] \setminus \Omega : H(\varphi) \leq B\} \sim c(k,G,H,\Omega) B^{a(H)} (\log B)^{b(k,H)-1}$$
	where $a(H)$ and $b(k,H)$ are as in Conjecture \ref{conj:balanced}, and 
	$$c(k,G,H,\Omega) = \frac{1}{|Z(G/M)^{\Gamma_k}|}
	\sum_{\substack{\psi \in \im(BG[k] \to B(G/M)[k]) \\
	BM_\psi(k) \not \subseteq \Omega}} c(k,M_\psi,H)$$
	with the resulting sum being convergent.
\end{conjecture}

Here we denote by $M_\psi$ the inner twist of $M$ by a lift of $\psi$ to $G$, in the sense of Example~\ref{ex:inner_twist}.

\begin{example}
	The main example of Conjecture \ref{conj:non_balanced} is counting $D_4$-quartics of
	bounded discriminant over $\Q$. Here Malle's conjecture is known \cite[Thm~1.3]{CDO02}, 
	with the leading constant
	\[\frac{1}{2} \sum_{D} \frac{2^{-i(D)}}{D^2}\zeta^*_{\Q(\sqrt{D})}(1)  
	\zeta_{\Q(\sqrt{D})}(2)^{-1}\] 
	where the sum is over fundamental discriminants $D$.
	This is given by an infinite sum of Euler products, as predicted by Conjecture \ref{conj:non_balanced}. More precisely, the minimal index conjugacy classes are given by 
	the class of the transpositions, so the Iitaka fibration is $BD_4 \to BC_2$. Thus Conjecture
	\ref{conj:non_balanced} predicts a sum over elements of $BC_2(\Q)$, i.e.~quadratic extensions,
	which is exactly what is obtained. For a verification that each summation is as predicted,
	i.e.~the leading constant of each fibre agrees with Conjecture \ref{conj:balanced},
	see \cite[\S10.2]{LS24}.
\end{example}

\begin{remark}
Wide classes of asymptotic formulae have been recently obtained for certain $G$ when counting by discriminant, using the so-called inductive method \cite{AOWW24, AB26, Cho25}. These results have a sum of leading constants over the Iitaka fibration, as predicted by Conjecture \ref{conj:non_balanced}. However the full Conjecture \ref{conj:non_balanced} is not typically verified, as one needs also to verify that the leading constant of each fibre agrees with the prediction from Conjecture \ref{conj:balanced}.
\end{remark}

\begin{question}
	Wright's method from \cite{Wri89} shows that when counting abelian fields of bounded discriminant, the leading constant is always given by a finite sum of Euler products, even when the discriminant is not balanced. The expression obtained this way is subtly different to ours, as it involves a combination of both the Brauer--Manin obstruction and M\"{o}bius inversion. 
	
	Is there any choice of $G$ and $H$ such that the leading constant in Conjecture \ref{conj:non_balanced} \textbf{cannot} be written as a finite sum of Euler products?
	
	We suspect a variant of Wright's argument should show that one obtains a finite sum of Euler products if the minimal weight elements generate $G/Z(G)$. It would be advantageous to have a rigorous proof of this and explore more generally height functions with this property, that one might call \emph{pseudo-balanced}.
\end{question}

\subsection{The total count}

Our main conjectures require the removal of a thin subset. This may seem unsatisfactory as it does not directly give a predicted asymptotic for the total count $\#\{ \varphi \in BG[k] : H(\varphi) \leq B\}$ or for the surjective cocycles.

However the framework is set up so that our conjectures can be applied to each thin subset removed to obtain an inductive description for the total count as a finite sum of asymptotic formulae. We briefly explain this here, with further details to be found in \cite[\S9.3]{LS24}.

Let $\Omega \subset BG[k]$ be a thin set. We assume that $\Omega$ is a finite union of basic thin sets, as in Definition \ref{def:thin}. Treating each piece separately, we may assume that $\Omega$ is itself a basic thin set, hence is the image of a map
$$BH[k] \to BG_{\varphi}[k]$$
where $H \subset G_{\varphi}$ is a proper $\Gamma_k$-subgroup of some inner twist of $G$. We therefore apply Conjecture \ref{conj:non_balanced} to the counting problem for $BH[k]$ to obtain a precise prediction for the number of cocycles of bounded height in $BH$. 

However extreme care is needed: firstly the automorphism groups when passing from $BH$ to $BG_{\varphi}$ may change, so one needs to carefully keep track of groupoid cardinalities. Secondly, as Conjecture \ref{conj:non_balanced} requires the removal of a thin set, it is theoretically possible that an additional thin set must be removed from $BH[k]$. This gives an inductive procedure for writing down a prediction for the total count.

In \cite[Rem.~9.9]{LS24} we show that, at least if the Galois action on $G$ is trivial, then this second inductive step is not required providing one is counting only surjective elements of $BG[k]$.

\subsubsection{The example $D_6$}
To illustrate the above procedure, we apply it in detail to the special case of $G = D_6 \cong C_2 \times S_3$ extensions. This will also allow us to demonstrate some of the general theory and results in the survey for an explicit example. An asymptotic formula for $D_6$-fields of bounded discriminant has been recently obtained in \cite{KLOST25}. We instead focus on ordering by radical discriminant $\rad \Delta$, which has the interesting feature that there are accumulating thin subfamilies of fields (a similar phenomenon happens for other dihedral groups, see \cite[\S10.4]{LS24}).

\begin{conjecture} \label{conj:D_6}
	\[
	\frac{1}{2}\#\{\varphi \in BD_6[\Q]: \rad \Delta(\varphi) \leq B\} \sim \left(\frac{35}{8748} + 
	\frac{5}{23328}\right)\prod_{p > 3}\left(1 - \frac{1}{p}\right)^5 \left(1 + \frac{5}{p}\right) B(\log B)^4.
	\]
	
	The leading constant takes the numerical value $0.0014623\dots$. Moreover, when $D_6$-extensions are ordered by radical discriminant then $3/59 \approx 5.08 \%$ contain the field $\Q(\mu_3)$.
\end{conjecture}

We now explain in detail how to obtain this prediction.  The group $D_6$ is \cite[\href{https://www.lmfdb.org/Groups/Abstract/12.4}{12.4}]{LMFDB}. All group-theoretic information we use can be read off from the LMFDB.

We have $\mathcal{M}(\rad \Delta) = \mathcal{C}_G^*$. The group $D_6$ has a rational character table with $5$ non-trivial conjugacy classes so $b(\Q, \rad \Delta) = 5$.

We have $\exp(G) = 6$ and $\Q(\mu_6) = \Q(\sqrt{-3})$. The possible accumulating thin subsets from \eqref{eqn:thin_set} are thus classified by maps $D_6 \to \Gal(\Q(\sqrt{-3})/\Q) \cong C_2$. There are three such non-trivial maps, two have kernel $S_3$ which has only two conjugacy classes so they cannot be accumulating.

Consider now the map $\pi: D_6 \to C_2$ with kernel $C_6$, this map has a section given by a transposition so there exists a map $\varphi: \Gamma_\Q \to \Gal(\Q(\sqrt{-3})/\Q) \to D_6$ such that $\pi(\varphi)$ defines the quadratic extension $\Q(\sqrt{-3})/\Q$.

The transpositions act by inversion on $C_6 \subset D_6$ so the inner twist $(C_6)_{\varphi}$ is isomorphic to $\mu_6$. The pullback height $\pi^* \rad \Delta$ is still anticanonical and $\mu_6(-1) = C_6$ so $b(\Q, \pi^* \rad \Delta) = 5$.

There are thus two leading constants contributing to the total count in this case. The first one from $G$ and the second one from $\mu_6$. Note that the map $B\mu_6[\Q] \to BD_6[\Q]$ is not injective. But the size of the fibers can be controlled by \cite[Lemma~2.14]{LS24} which shows that in terms of groupoid counts this map is $2 = |D_6/C_6|$ to $1$. As $Z(D_6) = C_2$ our Conjecture \ref{conj:balanced} predicts that 
\[
\frac{1}{2}\#\{\varphi \in BD_6[\Q]: \rad \Delta(\varphi) \leq B\} \sim \left(c(\Q,D_6,\rad \Delta) + \frac{1}{2}c(\Q, \mu_6, \pi^*(\rad \Delta))\right) B (\log B)^4.
\] 

We will now compute these leading constants separately. Note that the Brauer groups for both cases are trivial, this is \cite[Lem.~6.32]{LS24} for the algebraic part and \cite[Lem.~6.38(1, 3) ]{LS24} for the transcendental part. It remains to compute the local Tamagawa measures.
For $c(\Q,D_6,\rad \Delta)$ we use the isomorphism $D_6 \cong C_2 \times S_3$. 
\begin{enumerate}
	\item For $v = \infty$ we have $\tau_{\infty} = \frac{8}{12} = \frac{2}{3}$.
	\item For a prime $p > 3$ we have $\tau_p = 1 + \frac{5}{p}$ by the mass formula (Theorem \ref{thm:mass_formula}).
	\item If $p = 3$ then we use that $BD_6(\Q_3) = BC_2(\Q_3) \times BS_3(\Q_3)$ and use Lemma~\ref{lem:S_n} to identify each factor with a groupoid of \'etale algebras. We can then use the LMFDB \cite[\href{https://www.lmfdb.org/padicField/}{padicField}]{LMFDB} to enumerate such \'etale algebras. This gives
	\[
	\tau_3 = 1 + (2 \cdot \frac{1}{2}) \cdot \frac{1}{3} + (2 \cdot \frac{1}{2} + 3 \frac{1}{3} + 6 ) \cdot \frac{1}{3} + (2 \cdot \frac{1}{2}) \cdot (2 \cdot \frac{1}{2} + 3 \cdot \frac{1}{3} + 6 ) \cdot \frac{1}{3} = \frac{20}{3}.
	\]
	Here the summands denote the following contributions respectively: unramified, ramified in $BC_2$ and unramified in $BS_3$, unramified in $BC_2$ and ramified in $BS_3$, ramified in both $BC_2$ and $BS_3$.
	\item The prime $2$ can be dealt with in the same way the prime $3$. It gives 
	\[
	\tau_2 = 1 + (6 \cdot \frac{1}{2}) \cdot \frac{1}{2} + (6 \cdot \frac{1}{2} + 1 ) \cdot \frac{1}{2} + (6 \cdot \frac{1}{2}) \cdot (6 \cdot \frac{1}{2} + 1) \cdot \frac{1}{2} = \frac{21}{2}.
	\]
\end{enumerate}
The leading constant is thus
\[
c(\Q,D_6,\rad \Delta) = \frac{1}{2 \cdot 4 !} \cdot \frac{2}{3} \cdot \frac{21}{2} \cdot \frac{1}{2^5} \cdot \frac{20}{3} \cdot \frac{2^5}{3^5} \prod_{p > 3}\left(1 - \frac{1}{p}\right)^5 \left(1 + \frac{5}{p}\right).
\]
Note that the convergence factors come from $\zeta(s)^5$ and that $\zeta^*(1)=1$.
The map $B \mu_6 \to BD_{6}$ is given on the level of fields by the formula $\Q(\sqrt[6]{a}) \to \Q(\sqrt[6]{a}, \sqrt{-3})/\Q$. Note in particular that the image is always ramified at $3$. It is unramified at $2$ if and only if $a \equiv 1 \pmod 4$.

We can then compute the Tamagawa measures as
\begin{enumerate}
	\item At $\infty$ we have $\tau_{\infty} = \frac{\# \HH^1(\R, \mu_6)}{\# \HH^0(\R, \mu_6)} = 1$.
	\item For a prime $p > 3$ we can use the mass formula and we have $\tau_p = 1 + \frac{5}{p}$.
	\item For $p = 3$ we have $\tau_3 = \# \HH^1(\Q_3, \mu_6)/(3 \cdot \# \HH^0(\Q_3, \mu_6))= 6$ because $\Q_3^{\times} \cong \{3\}^{\Z} \times \F_3^{\times} \times (1 + 3\Z_3)$ and $(1 + 3\Z_3) \cong \Z_3$ by the $3$-adic logarithm.
	\item If $p = 2$ then we have $\HH^1(\Q_2, \mu_6) \cong \HH^1(\Q_2, \mu_2)$ because cubing is invertible on $\Q_2^{\times}$ by Hensel's lemma. We thus have $\tau_2 = (2 + \frac{6}{2})/\# \HH^0(\Q_2, \mu_6) = \frac{5}{2}$.
\end{enumerate}
The leading constant is thus
\[
c(\Q, \mu_6, \pi^*(\rad \Delta)) = \frac{1}{6 \cdot 4!} \cdot \frac{5}{2} \cdot \frac{1}{2^5} \cdot 6 \cdot \frac{2^5}{3^5} \prod_{p > 3}\left(1 - \frac{1}{p}\right)^5 \left(1 + \frac{5}{p}\right).
\]
Combining these two leading constants leads us to Conjecture \ref{conj:D_6}.

\subsection{Differences over function fields} 

Assume now that $k$ is the global function field of a smooth projective curve $X$ over the finite field $\F_q$. The theory of heights, Tamagawa measures and Brauer groups works exactly as in the case of number fields, assuming that $G$ is tame, i.e.~$\gcd(|G|,q)=1$. In what follows we make this assumption, progress on the wild case has been made in \cite{DYBM2,GS25,He26}, but there
is no theory of the Brauer group in the wild case, nor a prediction for the leading constant. We will assume that all local heights take values in $q^{\Z}$; this is true for natural height functions such as the discriminant
or radical discriminant.

A consequence is that the global height always takes values in $q^{\Z}$. The counting function $\#\{ \varphi \in BG[k] \setminus \Omega: H(\varphi) \leq B\}$ thus makes large jumps as $B$ crosses the values in $q^{\Z}$ and it is unreasonable to expect this function to have an asymptotic growth rate. We will instead study the counting function $N(G, H, q^m) := \#\{ \varphi \in BG[k] \setminus \Omega: H(\varphi) = q^{m}\}$. 

One new issue in the function field case is that there is no single leading constant, instead there is a certain periodic behavior in the limit. For example, the projective discriminant of a cover, given by $q^{\deg R}$ where $R$ is the discriminant divisor of the cover on the curve $X$, always has even exponent by the Riemann-Hurwitz formula. In this case $N(G, H, q^{m}) = 0$ for $m$ odd. 

One can extract a single leading constant from these periodic leading constants by averaging the leading constants. Note that if $B = q^{m}$ then $\log B = m \log q$ so the following Conjecture is analogous to Conjecture \ref{conj:balanced}.
\begin{conjecture}
	Assume that $H$ is balanced and $G$ is tame. Then
	\[
	\lim_{M \to \infty} \frac{1}{M}\sum_{m = 1}^M \frac{N(G, H, q^m)}{(m \log q)^{b(k, H) - 1} q^{a(H)m}} = (a(H) \log q) \cdot c(k, G, H) 
	\]
	where $a(H), b(k, H)$ and $c(k, G, H)$ are exactly as in Conjecture \ref{conj:balanced}.
\end{conjecture}
The explanation for the extra factor of $a(H) (\log q)$ is that the continuous analogue of $N(G, H, q^m)$ is the number of extensions whose height is in the interval $[B/q, B]$ as $q \to 1$.

This conjecture is proven in \cite{San25} for $k = \F_q(t)$, $G$ a $\Gamma_{\F_q}$-group and $H$ a height with everywhere good reduction under the assumption that $q$ is sufficiently large with respect to $|G|$. Indeed, even more is proven as it is shown that the leading constants have period dividing $|G|^2$ and an explicit formula for all leading constants is given. Note that in \cite[Thm.~1.2]{San25} there appear no factors of $\log q$, this is because they are absorbed in a different normalization of the Tamagawa measure.

\section{Equidistribution} \label{sec:equi}
In this section we study the problem of counting number fields with local conditions imposed. We call this \emph{equidistribution}; we use this terminology since, as we explain, the problem concerns the convergence of a sequence of measures.

\subsection{Weak equidistribution}

We first consider the problem of imposed finitely many local conditions, which we call \emph{(weak) equidistribution}. Recall that a subset of a measure space is called a \emph{continuity set} if its boundary has measure $0$.

\begin{conjecture}[Equidistribution] \label{conj:equi}
	Let $H$ be a balanced height and $W \subseteq \prod_v BG[k_v]$ be a continuity set.
	Then there exists a thin subset $\Omega \subset BG[k]$ such that
	$$\lim_{B \to \infty}\frac{\#\{ \varphi \in BG[k] \setminus \Omega : \varphi \in W, H(\varphi) \leq B\}}
	{\#\{ \varphi \in BG[k] \setminus \Omega : H(\varphi) \leq B\}} = 
	\frac{\tau_H(W \cap BG[\Adele_k]_{\mathcal{M}(H)}^{\Br})}
	{\tau_H( BG[\Adele_k]_{\mathcal{M}(H)}^{\Br})}.$$
\end{conjecture}

\begin{example} \label{ex:basic_open}
	A basic example of continuity set is one which is both open and closed. Since $BG[k_v]$ is discrete, these are relatively easy to construct. For example any subset of the form
	$$W = \prod_{v \in S} \{ \varphi_v \} \prod_{v \notin S} BG[k_v]$$
	for any finite set of places $S$ and $\varphi_v \in BG[k_v]$. In fact to prove weak equidistribution,
	it suffices to prove that it holds for subsets $W$ of this form \cite[Prop.~9.11]{LS24}. This result is proved by a standard topological argument, namely approximating arbitrarily well any continuity set by a collection of basic opens.
\end{example}

The terminology \emph{equidistribution} and the formulation in terms of continuity sets is justified by the following lemma, which makes clear that it concerns convergence of a sequence of measures.
We take into account the Brauer--Manin obstruction, as in Peyre \cite[3.21]{Pey21}, and also renormalise the Tamagawa measure to obtain a probability measure.

\begin{lemma} \label{lem:convergence}
	Let 
	$$\delta_{B} := \frac{1}{\#\{\varphi \in BG[k] \setminus \Omega : H(\varphi) \leq B\}}
	\sum_{\substack{\varphi \in  BG[k] \setminus \Omega \\
	H(\varphi) \leq B}} \delta_{\varphi}$$
	where $\delta_\varphi$ denotes the Dirac measure on $BG[\Adele_k]_{\mathcal{M}(H)}^{\Br}$ supported at $\varphi$. Then Conjecture \ref{conj:equi} holds if and only if $\delta_{B}$ converges weakly to the measure 
	$$\frac{\tau_{H}|_{BG[\Adele_k]_{\mathcal{M}(H)}^{\Br}}}{\tau_H(BG[\Adele_k]_{\mathcal{M}(H)}^{\Br})}, \quad \text{ as }B \to \infty.$$
\end{lemma}
\begin{proof}
	Immediate from the Portmanteau Theorem \cite[Thm.~13.16]{Klen14}.
\end{proof}

\subsection{Restricted ramification type}
In \cite[Conj~9.15]{LS24} we put forward a more general conjecture which imposes infinitely many well-behaved conditions, which we called \emph{strong equidistribution}. It concerns counting number fields of \emph{restricted ramification type}. By which we mean counting cocycles whose ramification type is restricted to lie in a proper subset of $\mathcal{C}_G$. We encode this condition using the large opens from Definition \ref{def:basic_open}, which have positive measure by Lemma \ref{lem:large_open}.

\begin{conjecture}[Strong equidistribution] \label{conj:restricted_ramification}
	Let $H$ be a big balanced height function on $BG$ and $W \subseteq BG(\Adele_k)_{\mathcal{M}(H)}$ a
	large open such that $W^{\Br} \neq \emptyset$.	Then there exists a thin subset $\Omega \subset BG[k]$ such that
	\begin{align*}
	    \frac{1}{|Z(G)^{\Gamma_k}|}&\#\{ \varphi \in BG[k] \setminus \Omega : H(\varphi)  \leq B, \varphi \in W\} \\ 
        &\sim c(k,G,H,W) B^{a(H)} (\log B)^{b(k,H)-1}
	\end{align*}
	where 
	$$
	c(k,G,H,W) 
	 = \frac{a(H)^{b(k,H)-1}\cdot |\Br^{\mathrm{e}}_{\mathcal{M}(H)} BG| \cdot \tau_{H}( W^{\Br})}
	{\#\dual{G}^{\Gamma_k} (b(k,H) - 1)!} \\
 	$$
 	where the Brauer--Manin obstruction is taken with respect to the Brauer group 
 	$\Br^{\mathrm{e}}_{\mathcal{M}(H)} BG$.
\end{conjecture}

Let us spell out explicitly what is being counted in Conjecture \ref{conj:restricted_ramification}. Let $W$ be a large open determined by sets $W_v \subseteq BG[k_v]$. The condition in Definition \ref{def:basic_open} that $BG(\O_v)_{\mathcal{M}(H)} \subseteq W_v$ means that we must count at least all minimally ramified cocycles at all but finitely many places. However the fact that we permit $W_v$ to be strictly smaller than $BG[k_v]$ means that we are allowed to exclude arbitrary other elements of $BG[k_v]$ at all places. For the special large opens from Example \ref{ex:restricted_ramification_type}, this corresponds to counting number fields with restricted ramification type.


By \cite[Prop.~9.16]{LS24}, to prove Conjecture \ref{conj:restricted_ramification} it actually suffices to consider the case where $W \subseteq BG(\Adele_k)_{\mathcal{M}(H)}$ is a basic open, in the sense of Definition \ref{def:basic_open}. This is proved via a standard topological argument of approximating any large open by basic opens. A similar argument also allows one to handle arbitrary continuity sets.

Here are some known cases of Conjecture \ref{conj:restricted_ramification}:
\begin{enumerate}
	\item $G = S_3,S_4,S_5$, discriminant height
	\cite[Thm.~1.1]{Bha14}, \cite[Thm.~2]{BSW}.
	Here the problem corresponds to counting degree $n$ fields with squarefree discriminant.
	\item $G = D_4$, $H$ the Artin conductor of the standard $2$-dimensional representation, where $\mathcal{M}(H) = \{(1,3), (1,2)(3,4)\}
	\subset D_4 \subset S_4$ the conjugacy classes of reflections
	\cite[Thm.~3]{ASVW21}.
	\item $G$ abelian with $H$  arbitrary;
	see \cite{Tav25} for the case of trivial Galois action and \cite{LT26}
	for non-trivial Galois action.
\end{enumerate}

There are not so many results in the literature on Conjecture \ref{conj:restricted_ramification}, most likely because the counting problem had not been previously rigorously formulated. It looks ripe for attacking in many new cases. In analogy with Manin's conjecture, Conjecture \ref{conj:restricted_ramification} should be interpreted as a problem concerning counting some version of integral points on $BG$.

\begin{remark}
	Let $\mathcal{C} \subseteq \mathcal{C}_{G}$ be Galois invariant such that 
	$\{ c \in \mathcal{C}^* : w(c) \text{ is minimal} \}$ generates $G$. 
	Conjecture \ref{conj:restricted_ramification} also allows one to obtain predictions for the a 
	priori more general problem of counting elements $\varphi \in BG[k]$ whose ramification type 
	is restricted to lie in $\mathcal{C}$, i.e.~such that $\varphi \in BG[\O_v]_{\mathcal{C}}$ for all
	$v$ outside some fixed set of places $S$. This follows by modifying the adelic height $H$ and 
	weight function such that its value on $\prod_{v \in S} BG[k_v] \prod_{v \notin S} 
	BG[\O_v]_{\mathcal{C}}$ is unchanged but now $\mathcal{M}(H) \subseteq \mathcal{C}$.
\end{remark}

\subsection{Strong approximation}

We emphasise that the following is an immediate consequence of Conjecture \ref{conj:restricted_ramification}, since Conjecture \ref{conj:restricted_ramification} implies that there is at least one cocycle in any basic open providing it satisfies the Brauer--Manin condition.

\begin{conjecture}[Strong Grunwald conjecture] \label{conj:dense}
	Let $\mathcal{C} \subset \mathcal{C}_G$ be Galois invariant and generate $G$. Then 
	$BG[k]$ is dense in $BG(\Adele_k)_{\mathcal{C}}^{{\Br^{\mathrm{e}}_{\mathcal{C}} BG}}$.
\end{conjecture}

This can be viewed as a qualitative version of Conjecture \ref{conj:restricted_ramification} and we expect to be of independent interest. It is a version of the Grunwald--Wang theorem for non-abelian groups and with restricted ramification type imposed. It is a very strong version of the inverse Galois problem.

On the Manin's conjecture side, for $\mathcal{C} = \mathcal{C}_G$ it is analogous to Colliot-Th\'{e}l\`{e}ne's conjecture that the Brauer--Manin obstruction is the only one to weak approximation on rationally connected varieties \cite[p.~174]{CT03}. For general $\mathcal{C}$, Conjecture \ref{conj:dense} should be interpreted as some version of Colliot-Th\'{e}l\`{e}ne's conjecture but for strong approximation on a non-proper variety.

We emphasise that the hypothesis that $\mathcal{C}$ generates $G$ is required in general in Conjecture \ref{conj:dense}. The following example corresponds to taking $\mathcal{C}$ empty; non-empty examples can be obtained by taking a suitable product of groups.

\begin{lemma} \label{lem:not_dense}
	Let $G$ be a non-trivial finite group with trivial abelianisation. Then $BG[k]$
	does not have dense image in $BG[\Adele_k]^{\Bre BG}.$
\end{lemma}
\begin{proof}
	We argue by contradiction.
	As the abelianization of $G$ is trivial, it follows from Lemma \ref{lem:algebraic_Brauer_H1_Pic}
	that $\Br^{\mathrm{e}}_1 BG = 0$. Thus by \eqref{seq:Br_extension} the group $\Bre BG$
	is finite. Hence by Theorem \ref{thm:trivial_on_O_v_R} there is a finite set of places $S$
	such that $W:=\prod_{v \in S}\{e\} \times 
	\prod_{v \notin S}BG(\O_v) \subset BG[\Adele_k]^{\Bre BG}.$ As $W$ is open
	we deduce that $BG[k] \cap W$ is dense in $W$. However $BG[k] \cap W$ is finite,
	since the elements may be represented by $G$-extensions of
	$k$ which are unramified outside of $S$, of which there are only finitely many by
	Hermite--Minkowski. As $W$ is infinite this gives a contradiction.
\end{proof}

\subsection{Bhargava's heuristic} \label{sec:Bhargava}
In an influential paper, Bhargava \cite{Bha07} put forward a series of conjectures on counting $S_n$ degree $n$ fields with local conditions imposed. He did this by trying to approximate the Dirichlet series for such number fields by an Euler product. Here we put forward a more general and corrected version of these heuristics in the form of a precise conjecture. Our changes are as follows:
\begin{enumerate}
	\item Rescaling of the zeta functions to obtain the correct leading constant. Bhargava specifically asks for the correct scaling in  \cite[(8.6)]{Bha07}.
	\item Removing a possible thin set of fields to avoid Kl\"{u}ners' style counter-examples.
	\item Taking a sum over elements of the corresponding Brauer group to account for Brauer--Manin obstructions, as in the Grunwald--Wang theorem.
	\item Restrict to balanced height functions only.
\end{enumerate}
	
This leads to the following.
Let $k$ be a number field, $G$ a finite $\Gamma_k$-group and $H$ a big balanced height on $BG$. Let $W \subseteq \prod_v BG[k_v]$ be a basic open set as in Definition \ref{def:basic_open}. We consider the height zeta function
$$\mathbf{Z}_{\Omega,W}(H,s) = \frac{1}{|Z(G)^{\Gamma_k}|}\sum_{\substack{\varphi \in BG[k] \\ \varphi \in W \\ \varphi \notin \Omega}} \frac{1}{H(\varphi)^s}$$ 
where $\Omega \subset BG[k]$ is a suitably large thin set. Height zeta functions first appeared in Manin's conjecture in the original paper \cite{FMT89}.  


\begin{conjecture}[Brauer spectral expansion] \label{conj:Brauer_spectral}
	There exists a meromorphic function $F_{H,W}(s)$ on $\re s \geq a(H)$ whose only pole lies at $s = a(H)$ where it has order less than $b(k,H)$, such that
	$$\mathbf{Z}_{\Omega,W}(H,s) = \frac{1}{\#\dual{G}^{\Gamma_k}} \sum_{\beta \in \Br^{\mathrm{e}}_{\mathcal{M}(H)} BG}
	\prod_{v} \left(\sum_{\varphi_v \in W_v} \frac{e^{2 \pi i \cdot \inv_v(\beta(\varphi_v))}}{|\Aut{\varphi_v}| H_v(\varphi_v)^{s}}\right) + F_{H,W}(s),$$
	on $\re > a(H).$ 
\end{conjecture}


Conjecture \ref{conj:Brauer_spectral} is inspired by a version of the Poisson summation formula for $BG$ when $G$ is abelian. In the abelian case it will be proven in \cite{LT26}. Hence why we view it as a spectral expansion. The Euler factors which appear here are modelled on the Brauer transforms from Definition \ref{def:Brauer_transform}.

In practice when counting, one obtains a sum over some characters rather than a sum over Brauer group elements. To put this into the setting of Conjecture \ref{conj:Brauer_spectral} one needs to interpret these as Brauer group elements. This philosophy appears in many places in the literature, for example in \cite[Thm.~4.5]{Lou18}, and nicely expounded upon in \cite[\S1.3]{CKR25}.

The main result in this section is that Conjecture \ref{conj:Brauer_spectral} is compatible with our counting conjectures.

\begin{theorem} \label{thm:Bhargava_heuristic}
	Conjecture \ref{conj:Brauer_spectral} implies Conjectures \ref{conj:balanced} and \ref{conj:equi}.
\end{theorem}


To approach this we begin by analysing in detail the Euler products which appear.

\begin{definition}
For $\beta \in \Bre BG$, we call the expression
$$\L_{H,W}(\beta,s):=\prod_{v} \left(\sum_{\varphi_v \in W_v} \frac{e^{2 \pi i \cdot \inv_v(\beta(\varphi_v))}}{|\Aut{\varphi_v}| H_v(\varphi_v)^{s}}\right), \quad \re s > a(H), $$
the \emph{$\L$-function of $\beta$}.
\end{definition}
We call this an $\L$-function for notational convenience to differentiate between various other zeta functions which will appear; we do not claim that it is an element of the Selberg class. Though we show later that it can be well-approximated by a product of Artin $\L$-functions (see Theorem \ref{thm:Brauer_zeta} and its proof).

Since $W$ is a basic open we have $W_v = BG[k_v]$ for all but finitely many $v$, hence the local Euler factors only depend on $\beta$ at all but finitely many places. The finitely many places where $W_v \neq BG[k_v]$ will not affect the overall analytic behaviour.

\begin{example} \label{ex:product_zeta}
	Our mass formula (Theorem \ref{thm:mass_formula}) shows without difficulty that
	$$\L_{H,W}(0,s) = \prod_{c \in \mathcal{C}_G^*/\Gamma_k} \zeta_{k(c)}(w(c)s) G_{H,W}(s)$$
	where $G_{H,W}(s)$ is given by an absolutely convergent Euler product on $\re s > a(H) - \varepsilon$
	for some $\varepsilon > 0$ which is non-zero at $s = a(H)$
	(this is a minor variant of Lemma~\ref{lem:Dedekind_zeta}). 
It follows immediately that $\L_{H,W}(0,s)$ has a pole of order $b(k,H)$ at $s = a(H)$,
	as expected from Conjecture \ref{conj:balanced}, and that the general Brauer $\L$-functions
	converge absolutely for $\re s > a(H)$.
\end{example}

We next analyse in detail the analytic properties of $\L_{H,W}(\beta,s)$. 

\begin{theorem} \label{thm:Brauer_zeta}
	Let $\beta \in \Bre BG$. 
	\begin{enumerate}
		\item The $\L$-function $\L_{H,W}(\beta,s)$ is meromorphic on $\re s > a(H) - \varepsilon$ for some $\varepsilon >0$ except for a possible pole at $s = a(H)$ of order at most $b(k,H)$. 
		\item If $\beta \notin \Br^{\mathrm{e}}_{\mathcal{M}(H)} BG$ then the order of the pole is less than $b(k,H)$.
		\item If $\beta \in \Br^{\mathrm{e}}_{\mathcal{M}(H)} BG$ then 
	$$\L_{H,W}(\beta,s) = \prod_{c \in \mathcal{C}_G^*/\Gamma_k} \zeta_{k(c)}(w(c) s) G_{H,W}(\beta,s)$$
	for some function $G(\beta,s)$ which is given by an absolutely convergent Euler product on $\re s \geq a(H) - \varepsilon$ which is non-zero at $s = a(H)$ when $\beta = 0$.
	\end{enumerate}
\end{theorem}
\begin{proof}
    We use a slightly more advanced version of our mass formula (Theorem \ref{thm:mass_formula}) for calculating these Brauer integrals, which we proved in \cite[Thm.~$8.23$]{LS24}. For all but finitely many tame places $v$ this reads:
	\begin{equation} \label{eqn:Brauer_mass_formula}
     \L_{H,W}(\beta,s)
     = 
 	\prod_{v}  \sum_{c \in \mathcal{C}_G^{\Gamma_{k_v}}} \frac{\chi_v(c)}{q_v^{w(c)s}}
	\end{equation}
	Here for fixed $c$ either $\chi_v(c)$ is zero or $\chi_v(c)$ is a frobenian function taking values in the roots of unity (it should be viewed as a kind of character sum).
	
	These functions $\chi_v(c)$ are defined in terms of the residue $\partial_{c}(\beta)$ of the Brauer group element, as in \S\ref{sec:residues}. In the most complicated case when the residue is \textit{transcendental}, we fortunately have $\chi_v(c) = 0$ so it does not contribute to the main singularity; essentially in this case $\chi_v(c)$ is a character sum which exhibits complete cancellation (this is our \cite[Lem.~8.22]{LS24}).
	
	When the residue is \emph{algebraic}, by Definition \ref{def:residue} it lies in the Galois cohomology group $\H^1(k(c), \Q/\Z)$, thus uniquely determines a Galois character $\psi_c: \Gamma_{k(c)} \to \Q/\Z$, which we view as a $1$-dimensional Artin representation via the embedding $\Q/\Z \to \C^\times, x \mapsto e^{2 \pi i x}$.
	
	Let $C$ be the $\Gamma_k$-Galois orbit of $c$. The Galois correspondence gives us a natural bijection between $C$ and $\Hom_k(k(c), k^{\sep})$. Composing with our fixed inclusion $k^{\sep} \subset k_v^{\sep}$ makes $c$ correspond to an embedding $k(c) \subset k_v^{\sep}$. The assumption that $c \in \mathcal{C}_G^{\Gamma_{k_v}}$ implies that the embedding factors through $k(c) \subset k_v$, in other words $c$ defines a split place $w$ of $k(c)$ lying over $v$. Then we define $\chi_v(c) :=  e^{2 \pi i \psi_c(\Frob_w)}$.
	
	Therefore, the formula \eqref{eqn:Brauer_mass_formula} implies that
	$$\L_{H,W}(\beta,s) = \prod_{\substack{ c \in \mathcal{C}_G^*/\Gamma_k \\ \partial_c(\beta) \textrm{ algebraic}}} \L_{k(c)}(\psi_c,w(c)s) G_{c,\beta,H,W}(s).$$
	Here $\psi_c$ is the Artin representation over $k(c)$ corresponding to the residue. This proves meromorphicity. Moreover it shows that the pole at $s=a(H)$ has order at most $b(k,H)$, hence shows (1).
	
	To determine which $\beta$ can contribute to the leading singularity, we note that this can happen if and only if the character $\psi_c$, hence the residue, is trivial for all $c \in \mathcal{M}(H)$. However by Lemma \ref{lem:residue} this happens if and only if $\beta \in  \Br^{\mathrm{e}}_{\mathcal{M}(H)} BG$. This proves (2) and (3)
\end{proof}

\begin{remark}
Theorem \ref{thm:Brauer_zeta} implies that if a version of Conjecture \ref{conj:Brauer_spectral} holds with a possible slightly bigger Brauer group, then Conjecture \ref{conj:Brauer_spectral} still holds. This highlights again the arithmetic significance of the partially unramified Brauer group elements.
\end{remark}

\begin{proof}[Proof of Theorem \ref{thm:Bhargava_heuristic}]
	To keep notation light, we denote the height zeta function by $\mathbf{Z}(s)$. 
	To prove Conjectures \ref{conj:balanced} and \ref{conj:equi}, 
	it suffices to obtain an asymptotic formula with the correct leading constant for $W$.
	By Conjecture \ref{conj:Brauer_spectral} and Theorem \ref{thm:Brauer_zeta} the height zeta function 
	admits a meromorphic continuation to $\re s \geq a(H)$ with 
	a pole of order at most $b(k,H)$. We do not claim that there is always a pole of exact 
	order $b(k,H)$ as there may be cancellation in the sum in Conjecture \ref{conj:Brauer_spectral};
	indeed this can happen if the local conditions imposed by $W$ are not realisable over $k$, i.e.~if
	there is a Brauer--Manin obstruction.
	
	In any case we apply a Tauberian theorem (e.g.~\cite[Thm.~A]{PTBZ25}) to deduce that
	\begin{align*}
	&\frac{1}{|Z(G)^{\Gamma_k}|}\#\{ \varphi \in BG[k] \setminus \Omega :H(\varphi)  \leq B, \varphi \in W_v \text{ for all }v\} \\
	& = (c + o(1)) B^{a(H)}(\log B)^{b(k,H)-1},
	\end{align*}
	where $$c= \frac{1}{a(H) ( b(k,H) -1)!} \lim_{s \to a(H)} (s - a(H))^{b(k,H)} \mathbf{Z}(s).$$
	It thus remains to show that this leading constant is as in Conjectures \ref{conj:balanced}
	and \ref{conj:equi}. Using Conjecture \ref{conj:Brauer_spectral} it suffices to show that
	$$\lim_{s \to a(H)} (s - a(H))^{b(k,H)} \mathbf{Z}(s) = 
	\frac{a(H)^{b(k,H)}|\Br^{\mathrm{e}}_{\mathcal{M}(H)} BG| \cdot\tau_H( W^{\Br})}{\# \dual{G}^{\Gamma_k}}.$$
	The factor $\# \dual{G}^{\Gamma_k}$ comes from the sum in Conjecture \ref{conj:Brauer_spectral}.  We may use Lemma \ref{lem:finite_sum_Euler_products} to write the Tamagawa measure as a finite sum of Euler products. Hence it suffices to show that for all $\beta \in \Br^{\mathrm{e}}_{\mathcal{M}(H)} BG$ we have
	\begin{equation} \label{eqn:limit_s_a}
	\lim_{s \to a(H)} (s - a(H))^{b(k,H)} \L_{H,W}(\beta,s)= a(H)^{b(k,H)}\hat{\tau}_H(\beta ; W),
	\end{equation}
	as in Definition \ref{def:Brauer_transform}.
	However, write
	$$
	\L_{H,W}(\beta,s) = \prod_{c \in \mathcal{M}(H)/\Gamma_k}\zeta_{k(c)}(a(H)^{-1} s)
	\times \L_{H,W}(\beta,s)\prod_{c \in \mathcal{M}(H)/\Gamma_k}\zeta_{k(c)}(a(H)^{-1} s)^{-1}.
	$$
	By Theorem \ref{thm:Brauer_zeta} the second factor is given by an absolutely
	convergent Euler product at $s = a(H)$. Using $\lim_{s \to a} (s - a)\zeta_k(a^{-1}s) = a\zeta_k^*(1)$, we find that the limit in \eqref{eqn:limit_s_a} equals
	\begin{align*}
	a(H)^{b(k,H)} \prod_{c \in \mathcal{M}(H)/\Gamma_k}\zeta_{k(c)}^*(1)
	\prod_v\left( \L_{H,W,v}(\beta, a(H))\prod_{c \in \mathcal{M}(H)/\Gamma_k}\zeta_{k(c),v}(1)^{-1}\right).
	\end{align*}
	Unravelling the definitions and using Lemma \ref{lem:Dedekind_zeta} proves \eqref{eqn:limit_s_a},
	as required.
\end{proof}



\section{Multi-heights} \label{sec:multi_heights}

\subsection{Definition}
Multi-heights are multi-variate height functions which provide a way to treat all heights simultaneously. They were systematically introduced into the study of Manin's conjecture by Peyre \cite[\S 4]{Pey21}, though already appear in earlier works on Manin's conjecture (see \S \ref{sec:multi_height_zeta}), and were subsequently introduced by Ellenberg--Venkatesh \cite[\S4.2]{EV05} and Gundlach \cite[\S 2]{Gun22} in the context of Malle's conjecture. 

We explain how to define multi-heights on $BG$. For each Galois orbit $\mathcal{C} \subset \mathcal{C}_G^*$, let  $w_{\mathcal{C}}$ be the weight function given by the indicator function of $\mathcal{C}$. Let $H_{\mathcal{C}}$ be a choice
of adelic height on $w_{\mathcal{C}}$. The weight functions $w_{\mathcal{C}}$ clearly give a $\Z$-basis for the space of all weight functions, thus $\prod_{\mathcal{C} \in \mathcal{C}_G^*/\Gamma_k} H_{\mathcal{C}}^{a_\mathcal{C}}$ gives a complete system of heights on every weight function for $a_{\mathcal{C}} \in \Z^{\mathcal{C}_G^*/\Gamma_k}$. We denote by $\mathbf{H}$ the tuple $(H_{\mathcal{C}})$, which is a multi-height on $BG$ in the sense of Peyre \cite[Def.~4.4]{Pey21}. 


\subsection{Conjectures}
In \cite{Gun22} Gundlach gives some heuristics on the distribution when counting with respect to multi-heights. Given Conjecture \ref{conj:balanced} we are now able to improve these to precise conjectures with a predicted leading constant. We let $H:= \prod_{\mathcal{C} \in \mathcal{C}_G^*/\Gamma_k} H_{\mathcal{C}}$ be the induced anticanonical height function.

\begin{conjecture}
	There exists a thin subset $\Omega \subset BG[k]$ such that
	$$\frac{1}{|Z(G)^{\Gamma_k}|}\#\{ \varphi \in BG[k] \setminus \Omega : H_{\mathcal{C}}(\varphi) \leq B_{\mathcal{C}}\} \sim 
	(b(k,H) - 1)! \cdot c(k,G,H) \prod_{\mathcal{C} \in \mathcal{C}_G^*/\Gamma_k} B_{\mathcal{C}}$$
	as $\min_{\mathcal{C} \in \mathcal{C}_G^*/\Gamma_k} B_{\mathcal{C}} \to \infty$, where
	$c(k,G,H)$ is as in Conjecture \ref{conj:balanced}.
\end{conjecture}

This conjecture has been verified in the case $G =D_4$ by Hansen and Zanoli \cite{HZ25}, provided the differences in the heights is not too large. We can moreover modify the conjecture to work with general boxes as in \cite[Question 4.8]{Pey21}.

\begin{conjecture} \label{conj:Peyre}
	Let $\mathcal{D} \subset \R^{\mathcal{C}_G^*/\Gamma_k}$ be a compact continuity set and $u \in  \R_{> 0}^{\mathcal{C}_G^*/\Gamma_k}$. For a real number $B > 1$ let $\mathcal{D}_B := \mathcal{D} + \log(B) \cdot u$. As $B \to \infty$ we have
	\begin{align*}
	&\frac{1}{|Z(G)^{\Gamma_k}|}\#\{ \varphi \in BG[k] : (\log H_{\mathcal{C}}(\varphi))_{\mathcal{C}} \in \mathcal{D}_B\}  \\ & \sim \nu(\mathcal{D}) (b(k,H) - 1)! c(k,G,H) B^{\sum_{\mathcal{C} \in \mathcal{C}_G^*/\Gamma_k} u_{\mathcal{C}}}
	\end{align*}
	where 
	\[
	\nu(\mathcal{D}) = \int_{\mathcal{D}} e^{\sum_{\mathcal{C} \in \mathcal{C}_G^*/\Gamma_k} x_{\mathcal{C}}} \prod_{\mathcal{C} \in \mathcal{C}_G^*/\Gamma_k} dx_{\mathcal{C}}.
	\]
\end{conjecture}

\begin{remark}\hfill
	\begin{enumerate} 
		\item Unlike Peyre, we do not need to remove a thin subset $\Omega$ in Conjecture \ref{conj:Peyre}, as the compactness of $\mathcal{D}$ ensures that the contribution of any thin subset to the right-hand side is $0$ for any sufficiently large $B$. Indeed, let $\varphi \in Z^1(k, G)$ and $M_{\varphi} \subsetneq G_{\varphi}$ a subgroup inducing a thin subfamily. 
		
		The map $\mathcal{C}_{M_{\varphi}} \to \mathcal{C}_G$ is not surjective by a theorem of Jordan \cite[Thm.~4']{Jor72}. For any $\mathcal{C} \subset \mathcal{C}_G$ not in the image we then have that the height $H_{\mathcal{C}}$ of the elements of the thin subset associated to $M_{\varphi}$ is bounded.
		\item Peyre only considers the above problem when $\mathcal{D}$ is a polyhedron, but this is equivalent to the more general continuity set case. Indeed, it suffices to show the above conjecture for any box $\mathcal{D} = \prod_{\mathcal{C} \in \mathcal{C}_G^*/\Gamma_k} [a_{\mathcal{C}}, b_{\mathcal{C}}]$, since any compact continuity set can be approximated arbitrarily well by a union of boxes.
	\end{enumerate}
\end{remark}


\subsection{Multi-height zeta functions} \label{sec:multi_height_zeta}

A natural way to approach multi-heights is via the \emph{multi-height zeta function}

$$\zeta_{\mathbf{H}}(\mathbf{s}) := \sum_{ \varphi \in BG[k] \setminus \Omega} \frac{1}{\mathbf{H}(\varphi)^{\mathbf{s}}} $$
for a suitable thin set $\Omega$ and a complex variable $\mathbf{s}$ in a suitable domain of $\C^{\mathcal{C}_G^*/\Gamma_k}$ for which the series converges.

Multi-height zeta functions were first used by Batyrev and Tschinkel \cite[\S4]{BT98t} in their proof of Manin's conjecture for toric varieties; they have subsequently appeared in numerous other works \cite{CT02,STBT07, Lou18}  on Manin's conjecture and related problems, particularly via harmonic analysis or spectral theory of automorphic forms. Their use in Malle's conjecture is due to Alberts and Bucur \cite{AB26}.

As in Conjecture \ref{conj:Brauer_spectral}, we expect the analytic properties of this multi-variate zeta function to be controlled by suitable multi-variate Brauer group integrals.


\end{document}